\documentclass[11pt,reqno]{amsart}
\usepackage{latexsym,esint,cite}
\usepackage{verbatim,wasysym}
\usepackage[left=2.8cm,right=2.8cm,top=2.9cm,bottom=2.9cm]{geometry}
\usepackage[latin1]{inputenc}
\usepackage{microtype}
\usepackage{color,enumitem,graphicx}
\usepackage[T1]{fontenc}
\usepackage{ae}	%sostituisce i font raster con font vettoriali
\usepackage[italian,UKenglish]{babel}    % utilizza la lingua italiana come predefinita
\usepackage{amsmath}
\usepackage{amsfonts}
\usepackage{amssymb}
\usepackage{amsthm}
\usepackage{algpseudocode}
\usepackage{float}
\usepackage{xcolor}
\usepackage{mathrsfs}
\usepackage{import}
\usepackage{geometry}

\usepackage{lettrine}   % per i capolettere
\usepackage{verbatim}
\usepackage{alltt}

\usepackage{graphics}
\usepackage{graphicx}

\graphicspath{{./imgs/}}

\usepackage{subfigure}
\usepackage{floatflt}

\usepackage[font=small,labelfont=bf]{caption}

\theoremstyle{plain}
\newtheorem{theorem}{Theorem}[section]

\newtheorem{definition}[theorem]{Definition}

\newtheorem{lemma}[theorem]{Lemma}
\newtheorem{remark}[theorem]{Remark}

\newtheorem{corollary}[theorem]{Corollary}
\numberwithin{equation}{section}
\theoremstyle{definition}

\newenvironment{claim}[1]{\par\noindent\textbf{Claim:} \space#1 \begin{itshape}
}{\end{itshape}}

\def\divergence{\mathop{\rm div}}
\def\dist{\mathop{\rm dist}}
\def\diam{\mathop{\rm diam}}
\newcommand{\R}{\ensuremath{\mathbb{R}}}

\title[Nonexistence to quasilinear parabolic equations in bounded domains]{Nonexistence of solutions to quasilinear parabolic equations with a potential in bounded domains}
\author{Giulia Meglioli}
\author{Dario D. Monticelli}
\author{Fabio Punzo}
\address[G.\ Meglioli]{Dipartimento di Matematica \newline\indent
	Politecnico di Milano, Milano, Italy}
\email{giulia.meglioli@polimi.it}

\address[D.D.\ Monticelli]{Dipartimento di Matematica \newline\indent
	Politecnico di Milano, Milano, Italy}
\email{dario.monticelli@polimi.it}

\address[F.\ Punzo]{Dipartimento di Matematica \newline\indent
	Politecnico di Milano, Milano, Italy}
\email{fabio.punzo@polimi.it}


\date{}

\begin{document}

\maketitle              % typeset the title of the contribution

\begin{abstract}
We are concerned with nonexistence results for a class of quasilinear parabolic differential problems with a potential in $\Omega\times(0,+\infty)$, where $\Omega$ is a bounded domain. In particular, we investigate how the behavior of the potential near the boundary of the domain and the power nonlinearity affect the nonexistence of solutions. Particular attention is devoted to the special case of the semilinear parabolic problem, for which we show that the critical rate of growth of the potential near
the boundary ensuring nonexistence is sharp.
\end{abstract}

\bigskip

\section{Introduction}
We investigate nonexistence of nonnegative, nontrivial global weak solutions to quasilinear parabolic inequalities of the following type:
\begin{equation}\label{problema}
\begin{cases}
\partial_t u - \divergence\left(|\nabla u|^{p-2}\nabla u\right)\ge V\, u^q &\quad \text{in}\,\, \Omega \times (0,+\infty) \\
\,\,u=0 \quad&\quad \text{on}\,\,\partial \Omega \times (0,+\infty) \\
\,\, u=u_0 \quad&\quad \text{in}\,\,\Omega\times\{0\}\,;
\end{cases}
\end{equation}
where $\Omega$ is an open bounded connected subset of $ \R^N$, $N\ge3$, $p>1$ and $q>\max\{p-1,1\}$. Furthermore, we assume that $V\in L^1_{{loc}}(\Omega\times[0,\infty))$, with $V>0$ a.e. in $\Omega\times(0,+\infty)$, and the initial condition satisfies $u_0\in L^1_{loc}(\Omega)$, with $u_0\ge 0$ a.e in $ \Omega$.

Global existence and finite time blow-up of solutions for problem \eqref{problema} has been deeply studied when $\Omega=\mathbb{R}^N$, see e.g. \cite{G1, G2, G3, Mitidieri, Mitidieri2, Poho, PT} and references therein. In particular, in \cite{Mitidieri2}, nonexistence of nontrivial weak solutions is proved for problem \eqref{problema} when $\Omega=\R^N$, $V\equiv 1$ and
$$
p>\frac{2N}{N+1},\quad\quad \max\{1,p-1\}<q\le p-1+\frac{p}{N}\,.
$$
%The method of proofs in \cite{Mitidieri2} mainly rely on proving some suitable local integral estimates based on the method of test functions and then study the asymptotic behavior of these estimates with respect to the relevant parameters of the problem.

Moreover, problem \eqref{problema} has been investigated also in the Riemannian setting, see e.g. \cite{BPT, MaMoPu, Punzo, Sun1, Zhang} and references therein.
%In \cite{Zhang}, it is studied nonexistence of nontrivial weak solutions to problem \eqref{problema} when $p=2$ and $\Omega=M$ where $M$ is a complete, noncompact Riemannian manifold.
In \cite{MaMoPu} problem \eqref{problema} is studied when $\Omega=M$ is a complete, $N$-dimensional, noncompact Riemannian manifold; it is investigated nonexistence of nonnegative nontrivial weak solutions depending on the interplay between the geometry of the underlying manifold, the power nonlinearity and the behavior of the potential at infinity, assuming that $u_0\in L^1_{loc}(M)$, $u\ge0$ a.e. in $M$ and $V\in L^1_{loc}(M\times[0,+\infty))$, $V>0$ a.e. in $M$.

Furthermore, we mention that nonexistence results of nonnegative nontrivial weak solutions have been also much investigated for solutions to elliptic quasilinear equation of the form
\begin{equation}\label{elliptic}
\frac{1}{a(x)}\divergence\left(a(x)|\nabla u|^{p-2}\nabla u\right)+ V(x) u^q\le 0 \quad \text{in}\,\,M\,,
\end{equation}
where
$$
a>0,\,\,\, a\in \text{Lip}_{loc}(M), \quad\,\,\, V>0\,\,\,\text{a.e. on}\,\, M, \,\,\, V\in L^1_{loc}(M),
$$
$p>1$, $q>p-1$ and $M$ can be the Euclidean space $\mathbb{R}^N$ or a general Riemannian manifold.

We refer to \cite{Dam, MitPoh1, MitPoh2, Mitidieri, Mitidieri2} for a comprehensive description of results related to problem \eqref{elliptic}, and also to more general problems, on $\R^N$. Problem \eqref{elliptic} when $M$ is a complete noncompact Riemannian manifold has been considered e.g. in \cite{Grig1, Grig2, MMP2, Sun2, Sun3}. In particular, in \cite{MMP2} the authors studied how the geometry of the underlying manifold $M$ and the behavior of the potential $V$ at infinity affect the nonexistence of nonnegative nontrivial weak solutions for inequality \eqref{elliptic}.
Finally, we mention that \eqref{elliptic} posed on an open relatively compact connected domain $\Omega \subset \mathbb{R}^N$ has been studied in \cite{MoPu}. Under the assumptions that
$$
a>0,\,\,\, a\in \text{Lip}_{loc}(\Omega), \quad\,\,\, V>0\,\,\,\text{a.e. on}\,\, \Omega, \,\,\, V\in L^1_{loc}(\Omega),
$$
$p>1$, $q>p-1$, the authors investigate the relation between the behavior of the potential $V$ at the boundary of $\Omega$ and nonexistence of nonnegative weak solutions.

\medskip

In the present paper, we are concerned with nonnegative weak solutions to problem \eqref{problema}. Under suitable weighted volume growth assumptions involving $V$ and $q$, we obtain nonexistence of global weak solutions (see Theorems \ref{teorema1}, \ref{teorema2}). The proofs are mainly based on the choice of a family of suitable test functions, depending on two parameters, that enables us to deduce first some appropriate a priori estimates, then that the unique global solution is $u\equiv 0$.
Such test functions are defined by adapting to the present situation those used in \cite{MaMoPu}; however, some important differences occur, since in \cite{MaMoPu} an unbounded underlying manifold is considered, whereas now we consider a bounded domain. In some sense, the role of {\it infinity} of \cite{MaMoPu} is now played by the boundary $\partial \Omega$. Obviously, this implies that such test functions satisfy different properties. To the best of our knowledge, the definition and use of such test functions are new.

As a special case, we consider in particular the semilinear parabolic problem
\begin{equation}\label{problema2}
\begin{cases}
\partial_t u - \Delta u = V u^q &\quad \text{in}\,\, \Omega \times (0,+\infty) \\
u=0 \quad&\quad \text{on}\,\,\partial \Omega \times (0,+\infty) \\
u=u_0\quad&\quad \text{in}\,\,\Omega\times\{0\}\,,
\end{cases}
\end{equation}
where $q>1$, $u_0\in L^1_{loc}(\Omega)$, $u_0\ge 0$ a.e. in $\Omega$, $V\in L^1_{loc}(\Omega\times[0,+\infty))$, with $V\ge0$, i.e. problem \eqref{problema} with $p=2$.

As a consequence of our general results, we infer that nonexistence of global solutions for problems \eqref{problema} and \eqref{problema2} prevails, when $$V(x,t)\ge Cd(x)^{-\sigma_1} \quad \text{for a.e. }\,\, x\in \Omega,\,t\in[0,+\infty)$$ for some $C>0$ and $$\sigma_1>q+1,$$ where
\begin{equation}\label{eq20}
d(x):=\dist(x,\partial \Omega)\quad\text{for any}\,\,\,x\in\overline{\Omega}.
\end{equation}
Furthermore, we show the sharpness of this result for the semilinear problem \eqref{problema2} in case $\partial\Omega$ is regular enough and $V=V(x)$ is continuous and independent of $t$. Indeed, under the assumption that
$$0\leq V(x)\le C d(x)^{-\sigma_1}\quad \text{for all }\,\, x\in \Omega$$
for some $C>0$ and
 $$0\leq\sigma_1<q+1,$$ we prove the existence of a global classical solution for problem \eqref{problema2} (see Theorem \ref{teorema3}), if the initial datum $u_0$ is small enough. This existence result is obtained by means of the sub-- and supersolution's method. In particular, we  construct a supersolution to problem \eqref{problema2}, which is actually a supersolution of the associated stationary equation. Such supersolution
is obtained as the fixed point of a suitable contraction map. In order to show that such a fixed point exists, we need to estimate some integrals involving the Green function associated to the Laplace operator $-\Delta$ in $\Omega$ (see Lemmas \ref{lemma4}, \ref{lemma5}). %and we prove that there exists $C>0$ such that
%$$
%0\le \int_{\Omega} G(x,y)d(y)^{\beta}\,dy \le C\, d(x)\,, \quad \text{for any}\,\,\, \beta>-2\,.
%$$
Finally, we study the {\it slightly supercritical} case
$$V(x,t)\ge d(x)^{-q-1}f(d(x))^{q-1}\quad \text{for a.e. }\,\, x\in \Omega,\,t\in[0,+\infty),$$
where $f$ is a function satisfying suitable assumptions and such that $\lim_{\varepsilon\rightarrow0^+}f(\varepsilon)=+\infty$, for which we prove nonexistence of nonnegative nontrivial weak solutions in $\Omega\times(0,+\infty)$. The proof of this result require a different argument with respect to the previous nonexistence results, which makes use of linearity of the operator and of the special form of the potential. Then the critical rate of growth $d(x)^{-q-1}$ as $x$ approaches $\partial\Omega$ is indeed sharp for the nonexistence of solutions to problem \eqref{problema2}.

The paper is organized as follows. In Section \ref{statements} we describe our main results and some consequences for problem \eqref{problema} (see Theorems \ref{teorema1}, \ref{teorema2} and Corollaries \ref{corollario1}, \ref{corollario3}); in particular in Subsection \ref{semilinearstatemets} we give the statements of our results for the semilinear problem \eqref{problema2} (see Theorems \ref{teorema3}, \ref{teorema4} and Corollary \ref{corollario4}). The definition of weak solutions and some preliminary results are stated in Section \ref{preliminaries}. Finally we prove the results obtained for problem \eqref{problema} in Sections \ref{proofnonlin} and \ref{proofnonlin2}, while the proofs of the results concerning the semilinear problem \eqref{problema2} are shown in Sections \ref{prooflin} and \ref{prooflin2}.

\section{Statements of the main results}\label{statements}
We now introduce the following two hypotheses (HP$1$) and (HP$2$) under which we will prove nonexistence of weak solutions for problem \eqref{problema}. Let $\theta_1\ge 1$, $\theta_2\ge1$, for each $\delta>0$  we define
\begin{equation}\label{eq21}
S:=\Omega\times[0,+\infty) \quad \text{and}\quad E_{\delta}:=\left\{(x,t)\in S\,:\,\, d(x)^{-\theta_2}+t^{\theta_1}\le \delta^{-\theta_2}\right\}.
\end{equation}
Moreover let
\begin{equation}\label{eq22}
\begin{aligned}
&\bar{s_1}:= \frac{q}{q-1}\theta_2\,, \quad \bar{s_2}:= \frac{1}{q-1}\,, \\
&\bar{s_3}:= \frac{pq}{q-p+1}\theta_2\,, \quad \bar{s_4}:= \frac{p-1}{q-p+1}\,.
\end{aligned}
\end{equation}

\begin{itemize}
\item[(HP$1$)] Assume that there exist constants $\theta_1\ge 1$, $\theta_2\ge1$, $C_0\ge0$, $C>0$, $\delta_0\in(0,1)$ and $\varepsilon_0>0$ such that
\begin{itemize}
\item[(i)]  for some $0<s_2<\bar{s_2}$
\begin{equation}\label{eq23}
\int_{E_{\frac{\delta}{2}}\setminus E_{\delta}} t^{(\theta_1-1)\left(\frac{q}{q-1}-\varepsilon\right)}V^{-\frac{1}{q-1}+\varepsilon} \,dx dt \,\le\, C\delta^{-\bar{s_1}-C_0\varepsilon} \left|\log(\delta)\right|^{s_2}
\end{equation}
for any $\delta\in(0,\delta_0)$ and for any $\varepsilon\in(0,\varepsilon_0)$;
\item[(ii)] for some $0<s_4<\bar{s_4}$
\begin{equation}\label{eq24}
\int_{E_\frac{\delta}{2}\setminus E_{\delta}} d(x)^{-(\theta_2+1)p\left(\frac{q}{q-p+1}-\varepsilon\right)}V^{-\frac{p-1}{q-p+1}+\varepsilon} \,dx dt \,\le\, C\delta^{-\bar{s_3}-C_0\varepsilon} \left|\log(\delta)\right|^{s_4}
\end{equation}
for any $\delta\in(0,\delta_0)$ and for any $\varepsilon\in(0,\varepsilon_0)$.
\end{itemize}
\end{itemize}

\begin{itemize}
\item[(HP$2$)] Assume that there exist constants $\theta_1\ge 1$, $\theta_2\ge1$, $C_0\ge0$, $C>0$, $\delta_0\in(0,1)$ and $\varepsilon_0>0$ such that
\begin{itemize}
\item[(i)]  for any $\delta\in(0,\delta_0)$ and for any $\varepsilon\in(0,\varepsilon_0)$
\begin{equation}\label{eq25}
\int_{E_\frac{\delta}{2}\setminus E_{\delta}} t^{(\theta_1-1)\left(\frac{q}{q-1}-\varepsilon\right)}V^{-\frac{1}{q-1}+\varepsilon} \,dx dt \,\le\, C\delta^{-\bar{s_1}-C_0\varepsilon} \left|\log(\delta)\right|^{\bar{s_2}},
\end{equation}
\begin{equation}\label{eq26}
\int_{E_\frac{\delta}{2}\setminus E_{\delta}} t^{(\theta_1-1)\left(\frac{q}{q-1}+\varepsilon\right)}V^{-\frac{1}{q-1}-\varepsilon} \,dx dt \,\le\, C\delta^{-\bar{s_1}-C_0\varepsilon} \left|\log(\delta)\right|^{\bar{s_2}};
\end{equation}
\item[(ii)] for any $\delta\in(0,\delta_0)$ and for any $\varepsilon\in(0,\varepsilon_0)$
\begin{equation}\label{eq27}
\int_{E_\frac{\delta}{2}\setminus E_{\delta}} d(x)^{-(\theta_2+1)p\left(\frac{q}{q-p+1}-\varepsilon\right)}V^{-\frac{p-1}{q-p+1}+\varepsilon} \,dx dt \,\le\, C\delta^{-\bar{s_3}-C_0\varepsilon} \left|\log(\delta)\right|^{\bar{s_4}},
\end{equation}
\begin{equation}\label{eq28}
\int_{E_\frac{\delta}{2}\setminus E_{\delta}} d(x)^{-(\theta_2+1)p\left(\frac{q}{q-p+1}+\varepsilon\right)}V^{-\frac{p-1}{q-p+1}-\varepsilon} \,dx dt \,\le\, C\delta^{-\bar{s_3}-C_0\varepsilon} \left|\log(\delta)\right|^{\bar{s_4}}.
\end{equation}
\end{itemize}

\end{itemize}
We can now state our main results.
\begin{theorem}\label{teorema1}
Let $p>1$, $q>\max\{p-1,1\}$, $V\in L^1_{loc}(\Omega\times[0,+\infty))$, $V>0$ a.e. in $\Omega\times(0,+\infty)$ and $u_0\in L^1_{loc}(\Omega)$, $u_0\ge 0$ a.e. in $\Omega$. Assume that condition (HP$1$) holds. If $u$ is a nonnegative weak solution of problem \eqref{problema}, then $u= 0$ a.e. in $S$.
\end{theorem}

\begin{theorem}\label{teorema2}
Let $p>1$, $q>\max\{p-1,1\}$, $V\in L^1_{loc}(\Omega\times[0,+\infty))$, $V>0$ a.e. in $\Omega\times(0,+\infty)$ and $u_0\in L^1_{loc}(\Omega)$, $u_0\ge 0$ a.e. in $\Omega$. Assume that condition (HP$2$) holds. If $u$ is a nonnegative weak solution of problem \eqref{problema}, then $u= 0$ a.e. in $S$.
\end{theorem}

As a consequence of Theorem \ref{teorema1} we introduce Corollary \ref{corollario1}. Let $d(x)$ be defined as in \eqref{eq20} and \eqref{eq21} respectively. Moreover we introduce functions $h:\Omega\to \R$ and $g:(0,+\infty)\to \R$ such that
\begin{align}
h(x)&\ge C \,d(x)^{-\sigma_1}\left(\log\left(1+d(x)^{-1}\right)\right)^{-\delta_1}& &\text{for a.e.}\,\,x\in\Omega, \label{eq210}\\
0<g(t)&\le C\,(1+t)^{\alpha}& &\text{for a.e.}\,\,t\in(0,+\infty),\label{eq211}
\end{align}
where $\sigma_1, \delta_1, \alpha \ge 0$, $C>0$. We can now state

\begin{corollary}\label{corollario1}
Let $p>1$, $q>\max\{p-1, 1\}$ and $u_0\in L^1_{loc}(\Omega)$, $u_0\ge 0$ a.e. in $\Omega$. Suppose that $V\in L^1_{loc}(\Omega\times[0,+\infty))$ satisfies
\begin{equation}\label{eq29}
V(x,t)\ge h(x)g(t)\quad\, \text{for a.e.}\,\,(x,t)\in S,
\end{equation}
where $h$ and $f$ satisfy \eqref{eq210} and \eqref{eq211} respectively. Moreover suppose that
\begin{equation}\label{eq212}
\begin{aligned}
&\int_0^T g(t)^{-\frac{1}{q-1}}\,dt \le CT^{\sigma_2} \left(\log T\right)^{\delta_2}, \\
&\int_0^T g(t)^{-\frac{p-1}{q-p+1}}\,dt \le CT^{\sigma_4},
\end{aligned}
\end{equation}
for $T>1$, $\sigma_2$, $\sigma_4$, $\delta_2\ge 0$ and $C>0$. Finally assume that
\begin{itemize}
\item[(i)] $\sigma_1>q+1$;
\item[(ii)] $0\le \sigma_2 \le \frac{q}{q-1}$;
\item[(iii)] $\delta_1< 1\,\,\,$ and $\,\,\,\delta_2<\frac{1-\delta_1}{q-1}$.
\end{itemize}
If $u$ is a nonnegative weak solution of problem \eqref{problema}, then $u= 0$ a.e. in $S$.
\end{corollary}

As an immediate consequence of Corollary \ref{corollario1}, choosing $g(t)\equiv1$, $\sigma_2=\sigma_4=1$ and $\delta_1=\delta_2=0$, we obtain the following

\begin{corollary}\label{corollario3}
Let $p>1$, $q>\max\{p-1, 1\}$ and $u_0\in L^1_{loc}(\Omega)$, $u_0\ge 0$ a.e. in $\Omega$. Suppose that $V\in L^1_{loc}(\Omega\times[0,+\infty))$ satisfies
\begin{equation}\label{eq36}
V(x,t)\ge Cd(x)^{-\sigma_1}\quad\, \text{for a.e.}\,\,(x,t)\in S,
\end{equation}
with $\sigma_1>q+1$. If $u$ is a nonnegative weak solution of problem \eqref{problema}, then $u= 0$ a.e. in $S$.
\end{corollary}

\subsection{Further result for semilinear problems}\label{semilinearstatemets} We prove, for the semilinear problem \eqref{problema2}, an existence result when $V=V(x)$ is continuous and independent of $t$ and $$0\leq V(x)\leq Cd(x)^{-\sigma_1}, \quad x\in \Omega,$$ with $$0\leq\sigma_1<q+1$$ (see Theorem \ref{teorema3}). Then we show a nonexistence result that yield that all nonnegative solutions of \eqref{problema2} are trivial if $V$ blows up at the boundary $\partial\Omega$ faster than $d(x)^{-q-1}$ (see Theorem \ref{teorema4} and Corollary \ref{corollario4} for precise statements).

\begin{theorem}\label{teorema3}
Suppose that $\partial\Omega$ is of class $C^3$ and let $u_0\in C(\Omega)$, $u_0\ge 0$ in $\Omega$, be such that there exists $\varepsilon>0$ such that
\begin{equation}\label{eq214}
0\le u_0\le \varepsilon \,d(x) \quad \text{for any} \,\,x\in\overline{\Omega}.
\end{equation}
Moreover let $V\in C(\Omega)$, $V\geq0$ in $\Omega$ and assume that for some $C>0$
\begin{equation}\label{eq215}
V=V(x) \le Cd(x)^{-\sigma_1}  \quad \text{for any} \,\,x\in\overline{\Omega}.
\end{equation}
with
\begin{equation}\label{eq216}
0\leq\sigma_1<q+1.
\end{equation}
Then problem \eqref{problema2} admits a classical solution $u$ in $\Omega\times(0,+\infty)$ if $\varepsilon>0$ is small enough.
\end{theorem}
For any $\varepsilon>0$ sufficiently small, set
\begin{equation}\label{eq616}
\Omega_\varepsilon=\{x\in\Omega\,|\,d(x)\geq\varepsilon\}.
\end{equation}

\begin{theorem}\label{teorema4}
Let $V\in L^1_{loc}(\Omega\times[0,\infty))$, $V>0$ a.e., and $u_0\in L^1_{loc}(\Omega)$, $u_0\ge 0$ a.e. Assume that there exists a nonincreasing function $f:(0,\varepsilon_0)\rightarrow[1,\infty)$ such that $\lim_{\varepsilon\rightarrow0^+}f(\varepsilon)=+\infty$ and such that, for some $C>0$,
for every $\varepsilon>0$ small enough
\begin{equation}\label{eq218}
\begin{aligned}
\int_0^{f(\varepsilon)}\int_{\Omega_\frac{\varepsilon}{2}\setminus\Omega_{\varepsilon}}V^{-\frac{1}{q-1}}\,dx dt\,\le\, C\,\varepsilon^{\frac{2q}{q-1}}\,,\\
\int_{\frac{1}{2}f(\varepsilon)}^{f(\varepsilon)}\int_{\Omega_\frac{\varepsilon}{2}}V^{-\frac{1}{q-1}}\,dx dt\,\le\, C\,f(\varepsilon)^\frac{q}{q-1}\,.
\end{aligned}
\end{equation}
If $u$ is a nonnegative weak solution of problem \eqref{problema2}, then $u= 0$ a.e. in $\Omega\times(0,+\infty)$.
\end{theorem}

As a consequence of Theorem \ref{teorema4} we have the following

\begin{corollary}\label{corollario4}
Suppose that $u_0\in L^1_{loc}(\Omega)$ with $u_0\ge 0$ a.e. in $\Omega$. Assume that $V$ satisfies for some $C>0$
\begin{equation}\label{eq37}
V(x,t)\ge Cd(x)^{-q-1}f(d(x))^{q-1}\quad\, \text{for a.e.}\,\,x\in\Omega,\, t\in[0,+\infty),
\end{equation}
where $f:(0,\operatorname{diam}(\Omega)]\rightarrow[1,+\infty)$ is nonincreasing in a right-neighborhood of $0$ and such that $\lim_{\varepsilon\rightarrow0^+}f(\varepsilon)=+\infty$. If $u$ is a nonnegative weak solution of problem \eqref{problema2}, then $u= 0$ a.e. in $\Omega\times(0,+\infty)$.
\end{corollary}

\begin{remark}
 We note that an example of function $f$ satisfying the assumptions of Corollary \ref{corollario4} is
 $$f(r)=\bigg[\overbrace{\log\circ\log\circ\ldots\circ\log}^{m\,\textrm{times}}\left(K+\frac{1}{r}\right)\bigg]^\beta, \quad r>0,$$
  for any $\beta>0$, $m\in\mathbb{N}$ and for $K>0$ sufficiently large.
\end{remark}

\section{Preliminaries}\label{preliminaries}

Let us first give the precise definition of weak solution to problem \eqref{problema} or \eqref{problema2}.
\begin{definition}\label{def}
Let $p>1$, $q>\max\{p-1,1\}$, $V\in L^1_{loc}(\Omega\times[0,+\infty))$, $V>0$ a.e. in $\Omega\times(0,+\infty)$ and $u_0\in L^1_{loc}(\Omega)$, $u_0\ge 0$ a.e. in $\Omega$. We say that $u\in W^{1,p}_{loc}(\Omega \times[0,+\infty)) \cap L^q_{loc}(\Omega\times[0,+\infty), Vdx dt)$ is a weak solution of problem \eqref{problema} if $u\ge 0$ a.e. in $\Omega\times(0,+\infty)$ and for every $\varphi \in \operatorname{Lip}(\Omega\times[0,\infty))$, $\varphi \ge 0$ in $\Omega \times[0,+\infty)$ and with compact support in $\Omega\times[0,\infty)$, one has
\begin{equation}\label{eq31}
\begin{aligned}
\int_0^{\infty}\int_{\Omega} Vu^q\,\varphi\,dx dt &\le \int_0^{\infty}\int_{\Omega}|\nabla u|^{p-2}\left \langle \nabla u, \nabla \varphi \right \rangle\, dx dt \\ &- \int_0^{\infty}\int_{\Omega}u\, \partial_t \varphi\, dx dt - \int_{\Omega} u_0\, \varphi(x,0) \, dx.
\end{aligned}
\end{equation}
\end{definition}

We now state some preliminary results that will be used in the proofs of Theorems \ref{teorema1} and \ref{teorema2}. We omit here the proofs, that can be found in \cite{MaMoPu}.

\begin{lemma}\label{lemma1}
Let $s\ge \max \left\{ 1, \frac{q}{q-1}, \frac{pq}{q-p+1}\right\}$ be fixed. Then there exists a constant $C>0$ such that for every $\alpha\in \big(-\min\big\{\frac{1}{2},\frac{p-1}{2}\big\}, 0 \big)$, for every nonnegative weak solution $u$ of problem \eqref{problema} and for every $\varphi\in Lip\left(\Omega \times [0,+\infty)\right)$ with compact support, $0\le \varphi \le 1$ one has
\begin{equation}\label{eq32}
\begin{aligned}
\frac 12\int_0^{\infty}\int_{\Omega} &V\,u^{q+\alpha}\,\varphi^s\,dx \,dt +\frac 34|\alpha| \int_0^{\infty}\int_{\Omega}|\nabla u|^{p}u^{\alpha-1}\,\varphi^s\, dx dt \\ & \le C\left\{ |\alpha|^{-\frac{(p-1)q}{q-p+1}} \int_0^{\infty}\int_{\Omega} |\nabla \varphi|^{\frac{p(q+\alpha)}{q-p+1}} V^{-\frac{p+\alpha-1}{q-p+1}} \, dx dt \right. \\
&\left.\quad\quad\,\,+ \int_{0}^{\infty} \int_{\Omega} |\partial_t \varphi| ^{\frac{q+\alpha}{q-1}} V^{-\frac{\alpha+1}{q-1}} \, dx\,dt\right\}.
\end{aligned}
\end{equation}
\end{lemma}

\begin{lemma}\label{lemma2}
Let $s\ge \max \left\{ 1, \frac{q+1}{q-1}, \frac{2pq}{q-p+1}\right\}$ be fixed. Then there exists a constant $C>0$ such that for every $\alpha\in \big(-\min\big\{\frac{1}{2},\frac{p-1}{2},\frac{q-1}{2},\frac{q-p+1}{2(p-1)}\big\}, 0 \big)$, for every nonnegative weak solution $u$ of problem \eqref{problema} and for every $\varphi\in Lip\left(S\right)$ with compact support and $0\le \varphi \le 1$ one has
\begin{equation}\label{eq33}
\begin{aligned}
\int_0^{\infty} & \int_{\Omega} V\,u^{q}\,\varphi^s\,dx \,dt \\
&\le C \left[ |\alpha|^{-1} \left(|\alpha|^{-\frac{(p-1)q}{q-p+1}} \int_0^{\infty}\int_{\Omega}V^{-\frac{p+\alpha-1}{q-p+1}}|\nabla \varphi |^{\frac{p(q+\alpha)}{q-p+1}}\, dx dt +  \int_0^{\infty}\int_{\Omega} V^{-\frac{\alpha+1}{q-1}}\,|\partial_t \varphi|^{\frac{q+\alpha}{q-1}}  \, dx dt \right)\right]^{\frac{p-1}{p}} \\&\times \left( \int\int_{S\setminus K} V\,u^{q}\varphi^s\,dx\,dt\right)^{\frac{(1-\alpha)(p-1)}{pq}} \left( \int\int_{S\setminus K} V^{-\frac{(1-\alpha)(p-1)}{q-(1-\alpha)(p-1)}}\,|\nabla\varphi|^{\frac{pq}{q-(1-\alpha)(p-1)}}\,dx dt\right)^{\frac{q-(1-\alpha)(p-1)}{pq}} \\
&+C\,\left(\int\int_{S\setminus K} V\,u^{q+\alpha}\,\varphi^s\,dx dt\right)^{\frac{1}{q+\alpha}}  \left(\int_0^{\infty}\int_{\Omega}V^{-\frac{1}{q+\alpha-1}}\, |\partial_t \varphi|^{\frac{q+\alpha}{q+\alpha-1}}\,dx dt\right)^{\frac{q+\alpha-1}{q+\alpha}},
\end{aligned}
\end{equation}
where $K:=\left\{(x,t)\in S\,: \, \varphi(x,t)=1\right\}$ and $S$ has been defined in \eqref{eq21}.
\end{lemma}

\begin{corollary}\label{corollario2}
Under the hypotheses of Lemma \ref{lemma2} one has
\begin{equation}\label{eq34}
\begin{aligned}
\int_0^{\infty} & \int_{\Omega} V\,u^{q}\,\varphi^s\,dx \,dt \\
&\le C \left[ |\alpha|^{-1} \left(|\alpha|^{-\frac{(p-1)q}{q-p+1}} \int_0^{\infty}\int_{\Omega}V^{-\frac{p+\alpha-1}{q-p+1}}|\nabla \varphi |^{\frac{p(q+\alpha)}{q-p+1}}\, dx dt +  \int_0^{\infty}\int_{\Omega} V^{-\frac{\alpha+1}{q-1}}\,|\partial_t \varphi|^{\frac{q+\alpha}{q-1}}  \, dx dt \right)\right]^{\frac{p-1}{p}} \\&\times \left( \int\int_{S\setminus K} V\,u^{q}\varphi^s\,dx\,dt\right)^{\frac{(1-\alpha)(p-1)}{pq}} \left( \int\int_{S\setminus K} V^{-\frac{(1-\alpha)(p-1)}{q-(1-\alpha)(p-1)}}\,|\nabla\varphi|^{\frac{pq}{q-(1-\alpha)(p-1)}}\,dx dt\right)^{\frac{q-(1-\alpha)(p-1)}{pq}} \\
&+C\,\left(|\alpha|^{-\frac{(p-1)q}{q-p+1}}\int_0^{\infty}\int_{\Omega} V^{-\frac{p+\alpha-1}{q-p+1}}\,|\nabla\varphi|^{\frac{p(q+\alpha)}{q-p+1}}\,dx dt + \int_0^{\infty}\int_{\Omega}V^{-\frac{\alpha+1}{q-1}}\, |\partial_t \varphi|^{\frac{q+\alpha}{q-1}}\,dx dt\right)^{\frac{1}{q+\alpha}} \\
&\times \left(\int_0^{\infty}\int_{\Omega}V^{-\frac{1}{q+\alpha-1}}\, |\partial_t \varphi|^{\frac{q+\alpha}{q+\alpha-1}}\,dx dt\right)^{\frac{q+\alpha-1}{q+\alpha}}.
\end{aligned}
\end{equation}
\end{corollary}

\begin{lemma}\label{lemma3}
Let $s\ge \max \left\{ 1, \frac{q+1}{q-1}, \frac{2pq}{q-p+1}\right\}$ be fixed. Then there exists a constant $C>0$ such that for every $\alpha\in \big(-\min\big\{\frac{1}{2},\frac{p-1}{2},\frac{q-1}{2},\frac{q-p+1}{2(p-1)}\big\}, 0 \big)$, for every nonnegative weak solution $u$ of problem \eqref{problema} and for every $\varphi\in Lip\left(S\right)$ with compact support and $0\le \varphi \le 1$ one has
\begin{equation}\label{eq35}
\begin{aligned}
\int_0^{\infty} & \int_{\Omega} V\,u^{q}\,\varphi^s\,dx \,dt \\
&\le C \left[ |\alpha|^{-1} \left(|\alpha|^{-\frac{(p-1)q}{q-p+1}} \int_0^{\infty}\int_{\Omega}V^{-\frac{p+\alpha-1}{q-p+1}}|\nabla \varphi |^{\frac{p(q+\alpha)}{q-p+1}}\, dx dt +  \int_0^{\infty}\int_{\Omega} V^{-\frac{\alpha+1}{q-1}}\,|\partial_t \varphi|^{\frac{q+\alpha}{q-1}}  \, dx dt \right)\right]^{\frac{p-1}{p}} \\&\times \left( \int\int_{S\setminus K} V\,u^{q}\varphi^s \,dx dt\right)^{\frac{(1-\alpha)(p-1)}{qp}} \left( \int\int_{S\setminus K} V^{-\frac{(1-\alpha)(p-1)}{q-(1-\alpha)(p-1)}}\,|\nabla\varphi|^{\frac{pq}{q-(1-\alpha)(p-1)}}\,dx dt\right)^{\frac{q-(1-\alpha)(p-1)}{pq}} \\
&+C\,\left(\int\int_{S\setminus K} V\,u^{q}\,\varphi^s \,dx\,dt\right)^{\frac{1}{q}}  \left(\int_0^{\infty}\int_{\Omega}V^{-\frac{1}{q-1}}\, |\partial_t \varphi|^{\frac{q}{q-1}}\,dx dt\right)^{\frac{q-1}{q}},
\end{aligned}
\end{equation}
where $K:=\left\{(x,t)\in S\,: \, \varphi(x,t)=1\right\}$ and $S$ has been defined in \eqref{eq21}.
\end{lemma}

\section{Proof of Theorem \ref{teorema1} and of Corollary \ref{corollario1}}\label{proofnonlin}
\begin{proof}[Proof of Theorem \ref{teorema1}]
For any $\delta>0$ sufficiently small, let $\alpha:=\frac{1}{\log \delta}$. Observe that $\alpha <0$ and $\alpha\to 0^-$ for $\delta\to 0$. We define for any $(x,t)\in S$
\begin{equation}\label{eq41}
\varphi(x,t):=\begin{cases}  \quad\quad\quad1 & \quad \text{in}\,\, E_{\delta} \\  \left[ \dfrac{d(x)^{-\theta_2}+t^{\theta_1}}{\delta^{-\theta_2}}\right]^{C_1\alpha} & \quad \text{in}\,\, \left(E_{\delta}\right)^C
\end{cases}.
\end{equation}
where
\begin{equation}\label{eq41b}
C_1>\frac{2(C_0+\theta_2+1)}{\theta_2q}
\end{equation}
with $C_0\geq0$, $\theta_1, \theta_2\ge1$ as in (HP$1$) and $E_{\delta}$ has been defined in \eqref{eq21}. Moreover, for any $n\in \mathbb{N}$ we define
\begin{equation}\label{eq42}
\eta_n(x,t):=\begin{cases}
\quad\quad\quad 1 & \quad \text{in}\,\, E_{\frac{\delta}{n}} \\
\frac{2^{\theta_2}}{2^{\theta_2}-1}-\frac{1}{2^{\theta_2}-1}\left(\dfrac{\delta}{n}\right)^{\theta_2}\left[ d(x)^{-\theta_2}+t^{\theta_1}\right] & \quad \text{in}\,\, E_{\frac{\delta}{2n}}\setminus E_{\frac{\delta}{n}} \\
\quad\quad\quad 0 & \quad \text{in}\,\, E_{\frac{\delta}{2n}}^C
\end{cases}.
\end{equation}
Let
\begin{equation}\label{eq43}
\varphi_n(x,t):= \eta_n(x,t)\,\varphi(x,t).
\end{equation}
Observe that $\varphi_n\in \operatorname{Lip}(S)$ and $0\le \varphi \le 1$. Moreover, for any $a\ge 1$ we have
\begin{equation}\label{eq44}
|\partial_t \varphi_n|^a=|\eta_n\partial_t\varphi\,+\,\varphi\partial_t\eta_n|^a \le 2^{a-1}\left(|\partial_t\varphi|^a + \varphi^a|\partial_t\eta_n|^a\right).
\end{equation}
\begin{equation}\label{eq45}
|\nabla \varphi_n|^a=|\eta_n\nabla\varphi\,+\,\varphi\nabla\eta_n|^a \le 2^{a-1}\left(|\nabla\varphi|^a + \varphi^a|\nabla\eta_n|^a\right).
\end{equation}
Let $s\ge \max\left\{1,\,\frac{q}{q-1},\,\frac{pq}{q-p+1}\right\}$, we apply Lemma \ref{lemma1} with $\varphi$ replaced by the family of functions $\varphi_n$. Then, for some positive constant $C$, for every $n\in \mathbb{N}$ and $|\alpha|>0$ small enough we have
$$
\begin{aligned}
\int_0^{\infty}\int_{\Omega} &V\,u^{q+\alpha}\,\varphi_n^s\,dx \,dt  \\
&\le C\left\{|\alpha|^{-\frac{(p-1)q}{q-p+1}} \int_0^{\infty}\int_{\Omega}|\nabla \varphi_n|^{\frac{p(q+\alpha)}{q-p+1}}V^{-\frac{p+\alpha-1}{q-p+1}}\, dx dt +  \int_0^{\infty}\int_{\Omega}|\partial_t \varphi_n|^{\frac{q+\alpha}{q-1}}V^{-\frac{\alpha+1}{q-1}}\, dx dt\right\}\\
& \le C |\alpha|^{-\frac{(p-1)q}{q-p+1}} \left[\int_0^{\infty}\int_{\Omega} |\nabla \varphi|^{\frac{p(q+\alpha)}{q-p+1}} V^{-\frac{p+\alpha-1}{q-p+1}} \, dx dt  + \int_{0}^{\infty} \int_{\Omega}  \varphi ^{\frac{p(q+\alpha)}{q-p+1}} |\nabla \eta_n|^{\frac{p(q+\alpha)}{q-p+1}}V^{-\frac{p+\alpha+1}{q-p+1}} \, dx\,dt\right] \\
&+ C\left[\int_0^{\infty}\int_{\Omega}|\partial_t \varphi|^{\frac{q+\alpha}{q-1}}V^{-\frac{\alpha+1}{q-1}}\, dx dt + \int_0^{\infty}\int_{\Omega} \varphi^{\frac{q+\alpha}{q-1}} |\partial_t \eta_n| ^{\frac{q+\alpha}{q-1}}V^{-\frac{\alpha+1}{q-1}}\, dx dt\right].
\end{aligned}
$$
Let us define
\begin{equation}\label{eq45b}
\tilde{E}_{\delta,n}:=E_\frac{\delta}{2n} \setminus E_{\frac{\delta}{n}},
\end{equation}
and
\begin{align}
&\label{eq46}I_1:=\int_0^{\infty}\int_{\Omega} |\nabla \varphi|^{\frac{p(q+\alpha)}{q-p+1}} V^{-\frac{p+\alpha-1}{q-p+1}} \, dx dt,
\\
&\label{eq47}
I_2:= \int\int_{\tilde{E}_{\delta,n}}  \varphi ^{\frac{p(q+\alpha)}{q-p+1}} |\nabla \eta_n|^{\frac{p(q+\alpha)}{q-p+1}}V^{-\frac{p+\alpha+1}{q-p+1}} \, dx\,dt,
\\
&\label{eq48}
I_3:=\int_0^{\infty}\int_{\Omega}|\partial_t \varphi|^{\frac{q+\alpha}{q-1}}V^{-\frac{\alpha+1}{q-1}}\, dx dt,
\\
&\label{eq49}
I_4:= \int\int_{\tilde{E}_{\delta,n}} \varphi^{\frac{q+\alpha}{q-1}} |\partial_t \eta_n| ^{\frac{q+\alpha}{q-1}}V^{-\frac{\alpha+1}{q-1}}\, dx dt.
\end{align}
Then the latter inequality can be read, for a positive constant $C$ and for every $n\in \mathbb{N}$, as
\begin{equation}\label{eq49b}
\int_0^{\infty}\int_{\Omega} V\,u^{q+\alpha}\,\varphi_n^s\,dx dt \le C |\alpha|^{-\frac{(p-1)q}{q-p+1}} \left[ I_1+I_2\right]+C\left[I_3+I_4\right].
\end{equation}

In view of \eqref{eq41} and \eqref{eq42}, for $|\alpha|>0$ small enough noand for every $n\in \mathbb{N}$, we have
\begin{equation}\label{eq410}
\begin{aligned}
I_2&\le \int\int_{\tilde{E}_{\delta,n}}  C \,n ^{C_1\alpha\theta_2\frac{p(q+\alpha)}{q-p+1}} \left(\frac{\delta}{n}\right)^{\theta_2\frac{p(q+\alpha)}{q-p+1}} \left[d(x)^{-\theta_2-1}|\nabla d(x)|\right]^{\frac{p(q+\alpha)}{q-p+1}}V^{-\frac{p+\alpha+1}{q-p+1}} \, dx\,dt\\
&\le C\,n ^{\theta_2\frac{p(q+\alpha)}{q-p+1}(C_1\alpha-1)}\delta^{\theta_2\frac{p(q+\alpha)}{q-p+1}}  \int\int_{\tilde{E}_{\delta,n}} d(x)^{-(\theta_2+1)\frac{p(q+\alpha)}{q-p+1}} V^{-\frac{p+\alpha+1}{q-p+1}} \, dx\,dt.
\end{aligned}
\end{equation}
Due to assumption (HP$1)-(ii)$ with $\varepsilon=-\frac{\alpha}{q-p+1}>0$, \eqref{eq410} reduces to
\begin{equation}\label{eq411}
I_2\le C\,n ^{\theta_2\frac{p(q+\alpha)}{q-p+1}(C_1\alpha-1)}\delta^{\theta_2\frac{p(q+\alpha)}{q-p+1}}\left(\frac{\delta}{n}\right)^{-\frac{pq\theta_2}{q-p+1}-C_0\varepsilon}\,\left|\log\left(\frac{\delta}{n}\right)\right|^{s_4} ,
\end{equation}
with $s_4$ as in (HP$1$). Now observe that, due \eqref{eq41b}, we have
$$
\frac{|\alpha|}{q-p+1}\left(-\theta_2\,p+C_1\,p\,\theta_2(q+\alpha)-C_0\right) \ge \frac{|\alpha|}{q-p+1}.
$$
Moreover, there exist $\bar C>0$ such that
$$
\delta^{\frac{\alpha}{q-p+1}\left[\theta_2 p+C_0\right]}=\mathrm{e}^{\frac{\alpha}{q-p+1}\left[\theta_2 p+C_0\right]\log(\delta)}=\mathrm{e}^{\frac{\theta_2 p+C_0}{q-p+1}}\le \bar C.
$$
Then from \eqref{eq411} we deduce, for some $C>0$ and $|\alpha|>0$ small enough
\begin{equation}\label{eq412}
I_2\le C\,n^{-\frac{|\alpha|}{q-p+1}}\,\left|\log\left(\frac{\delta}{n}\right)\right|^{s_4} .
\end{equation}

Similarly, in view of \eqref{eq41} and \eqref{eq42}, for $|\alpha|>0$ small enough and for every $n\in \mathbb{N}$ we have
\begin{equation}\label{eq413}
\begin{aligned}
I_4&\le C\int\int_{\tilde{E}_{\delta,n}} n^{\theta_2C_1\alpha\left(\frac{q+\alpha}{q-1}\right)}\left(\frac{\delta}{n}\right)^{\theta_2\left(\frac{q+\alpha}{q-1}\right)} t^{(\theta_1-1)\frac{q+\alpha}{q-1}}\,V^{-\frac{\alpha+1}{q-1}}\,dx dt \\
&\le C n^{\theta_2\left(\frac{q+\alpha}{q-1}\right)(C_1\alpha-1)}\delta^{\theta_2\left(\frac{q+\alpha}{q-1}\right)} \int\int_{\tilde{E}_{\delta,n}} t^{(\theta_1-1)\left(\frac{q+\alpha}{q-1}\right)}\,V^{-\frac{\alpha+1}{q-1}}\,dx dt .
\end{aligned}
\end{equation}
Due to assumption HP$1(i)$ with $\varepsilon=-\frac{\alpha}{q-1}>0$, \eqref{eq413} reduces to
\begin{equation}\label{eq414}
\begin{aligned}
I_4&\le C \,n^{\theta_2\left(\frac{q+\alpha}{q-1}\right)(C_1\alpha-1)}\delta^{\theta_2\left(\frac{q+\alpha}{q-1}\right)} \left(\frac{\delta}{n}\right)^{-\frac{q}{q-1}\theta_2-C_0\varepsilon}\left|\log\left(\frac{\delta}{n}\right)\right|^{s_2} \\
&\le C\,n^{\frac{1}{q-1}\left[C_1\alpha\theta_2(q+\alpha)-\alpha\theta_2+C_0|\alpha|\right]}\,\delta^{\frac{1}{q-1}\left[\alpha\theta_2+C_0\alpha\right]}\left|\log\left(\frac{\delta}{n}\right)\right|^{s_2},
\end{aligned}
\end{equation}
with $s_2$ as in (HP$1$). We now observe that, due to \eqref{eq41b}, we can write
\begin{equation}\label{eq415}
n^{-\frac{|\alpha|}{q-1}\left[C_1\theta_2(q+\alpha)-\theta_2-C_0\right]}\le n^{-\frac{|\alpha|}{q-1}}.
\end{equation}
Moreover, observe that there exist $\bar{C}>0$ such that
\begin{equation}\label{eq416}
\delta^{\frac{\alpha}{q-1}\left(\theta_2+C_0\right)}=\mathrm{e}^{\frac{\alpha}{q-1}\left(\theta_2+C_0\right)\log(\delta)}=\mathrm{e}^{\frac{\theta_2+C_0}{q-1}}\le \bar{C}.
\end{equation}
By plugging \eqref{eq415} and \eqref{eq416} into \eqref{eq414} we get for $\delta>0$ small enough
\begin{equation}\label{eq417}
I_4\le C\,n^{-\frac{|\alpha|}{q-1}}\left|\log\left(\frac{\delta}{n}\right)\right|^{s_2} .
\end{equation}
Let us now consider integral $I_1$ defined in \eqref{eq46}. By using the definition of $\varphi$ in \eqref{eq41} we can write
\begin{equation}\label{eq418}
\begin{aligned}
I_1&\le \int\int_{E_{\delta}^C} \left[C_1|\alpha|\theta_2\left(\frac{d(x)^{-\theta_2}+t^{\theta_1}}{\delta^{-\theta_2}}\right)^{C_1\alpha-1}\frac{d(x)^{-\theta_2-1}}{\delta^{-\theta_2}}\right]^{\frac{p(q+\alpha)}{q-p+1}} V^{-\frac{p+\alpha-1}{q-p+1}}\,dx dt\\
&\le C\int\int_{E_{\delta}^C} |\alpha|^{\frac{p(q+\alpha)}{q-p+1}}\left[d(x)^{-\theta_2}+t^{\theta_1}\right]^{\frac{(C_1\alpha-1)p(q+\alpha)}{q-p+1}}d(x)^{-\frac{(\theta_2+1)p(q+\alpha)}{q-p+1}}\delta^{\frac{\theta_2C_1\alpha p(q+\alpha)}{q-p+1}} V^{-\frac{p+\alpha-1}{q-p+1}}\,dx dt.
\end{aligned}
\end{equation}
Similarly to \eqref{eq416}, we can say that there exist $\bar{C}>0$ such that
$$
\delta^{\frac{\theta_2C_1\alpha p(q+\alpha)}{q-p+1}}\le \bar{C},
$$
hence \eqref{eq418}, for some constant $C>0$, reduces to
\begin{equation}\label{eq419}
I_1\le C|\alpha|^{\frac{p(q+\alpha)}{q-p+1}} \int\int_{E_{\delta}^C} V^{-\frac{p+\alpha-1}{q-p+1}} d(x)^{-\frac{(\theta_2+1)p(q+\alpha)}{q-p+1}}\left[\left(d(x)^{-\theta_2}+t^{\theta_1}\right)^{-\frac{1}{\theta_2}}\right]^{-\frac{\theta_2(C_1\alpha-1)p(q+\alpha)}{q-p+1}} \,dx dt.
\end{equation}

\begin{claim}{} If $f:(0,+\infty)\to[0,+\infty)$ is a non decreasing function and if (HP$1)-(ii)$ holds then, for any $0<\varepsilon<\varepsilon_0$ and for any $\delta>0$ small enough, we can write
\begin{equation}\label{eq420}
\begin{aligned}
\int\int_{E_{\delta}^C} f\left(\left[\left(d(x)^{-\theta_2}+t^{\theta_1}\right)^{-\frac{1}{\theta_2}}\right]\right) &d(x)^{-(\theta_2+1)p\left(\frac{q}{q-p+1}-\varepsilon\right)}V^{-\frac{p-1}{q-p+1}+\varepsilon}\,dx dt \\
& \le C \int_0^{2\delta} f(z) z^{-\frac{pq}{q-p+1}\theta_2-C_0\varepsilon-1}|\log z|^{s_4}\,dz,
\end{aligned}
\end{equation}
for some constant $C>0$.
\end{claim}

To show the claim, we first observe that
$$
f\left(\big(d(x)^{-\theta_2}+t^\theta_1\big)^{-\frac{1}{\theta_2}}\right)\le f\left(\frac{\delta}{2^n}\right)\quad \quad\text{in}\,\,\, E_{\frac{\delta}{2^{n+1}}}\setminus E_{\frac{\delta}{2^n}}.
$$
Hence, due to HP$1(ii)$, we can write
$$
\begin{aligned}
\int\int_{(E_{\delta})^C} &f\left(\left[d(x)^{-\theta_2}+t^{\theta_1}\right]^{-\frac{1}{\theta_2}}\right) d(x)^{-(\theta_2+1)p\left(\frac{q}{q-p+1}-\varepsilon\right)}V^{-\frac{p-1}{q-p+1}+\varepsilon}\,dx dt \\
&=\sum_{n=0}^{+\infty} \int\int_{E_{\frac{\delta}{2^{n+1}}}\setminus E_{\frac{\delta}{2^n}}} f\left(\left[d(x)^{-\theta_2}+t^{\theta_1}\right]^{-\frac{1}{\theta_2}}\right)d(x)^{-(\theta_2+1)p\left(\frac{q}{q-p+1}-\varepsilon\right)}V^{-\frac{p-1}{q-p+1}+\varepsilon}\,dx dt \\
&\le \sum_{n=0}^{+\infty} f\left(\frac{\delta}{2^n}\right)\int\int_{E_{\frac{\delta}{2^{n+1}}}\setminus E_{\frac{\delta}{2^n}}}d(x)^{-(\theta_2+1)p\left(\frac{q}{q-p+1}-\varepsilon\right)}V^{-\frac{p-1}{q-p+1}+\varepsilon}\,dx dt \\
&\le C \sum_{n=0}^{+\infty} f\left(\frac{\delta}{2^{n}}\right)\left(\frac{\delta}{2^{n}}\right)^{-\frac{pq}{q-p+1}\theta_2-C_0\varepsilon}\left|\log\left(\frac{\delta}{2^{n}}\right)\right|^{s_4} \\
&\le C\sum_{n=0}^{+\infty} \int_{\frac{\delta}{2^{n}}}^{\frac{\delta}{2^{(n-1)}}}f(z) z^{-\frac{pq}{q-p+1}\theta_2-C_0\varepsilon-1}|\log z|^{s_4}\,dz\\
&=C\int_{0}^{2\delta}f(z) z^{-\frac{pq}{q-p+1}\theta_2-C_0\varepsilon-1}|\log z|^{s_4}\,dz.
\end{aligned}
$$

We now apply \eqref{eq420} with $\varepsilon=\frac{|\alpha|}{q-p+1}>0$ to inequality \eqref{eq419}. We get
\begin{equation}\label{eq421}
I_1\le C|\alpha|^{\frac{p(q+\alpha)}{q-p+1}}\int_0^{2\delta}z^{-\theta_2\frac{(C_1\alpha-1)p(q+\alpha)}{q-p+1}-\frac{pq}{q-p+1}\theta_2+\frac{C_0\alpha}{q-p+1}-1}|\log z|^{s_4}\,dz.
\end{equation}
We define
\begin{equation}\label{eq422}
b:=\frac{1}{q-p+1}\left(-\theta_2C_1\alpha p(q+\alpha)+\theta_2p\alpha+C_0\alpha\right),
\end{equation}
and due to \eqref{eq41b}, we observe that
$$
b \ge \frac{|\alpha|}{q-p+1} > 0.
$$
By plugging \eqref{eq422} into inequality \eqref{eq421} we can write
\begin{equation}\label{eq423}
I_1\le C\,|\alpha|^{\frac{p(q+\alpha)}{q-p+1}}\int_0^{2\delta}z^{b-1} |\log z|^{s_4}\,dz.
\end{equation}
Let us now perform a change of variable, we define
$$
y:=b\log z,
$$
hence from \eqref{eq423} we deduce
\begin{equation}\label{eq424}
\begin{aligned}
I_1&\le C\,|\alpha|^{\frac{p(q+\alpha)}{q-p+1}}b^{-s_4-1}\int_{-\infty}^0 e^y |y|^{s_4}\,dy \\
&\le C\,|\alpha|^{\frac{p(q+\alpha)}{q-p+1}}\left(\frac{|\alpha|}{q-p+1}\right)^{-s_4-1}\\
&\le C\,|\alpha|^{\frac{pq}{q-p+1}-s_4-1}.
\end{aligned}
\end{equation}
for $|\alpha|>0$ small enough, with $s_4$ as in (HP$1)-(ii)$.

Finally, let us consider $I_3$ defined in \eqref{eq48}. Due to the definition of $\varphi$ in \eqref{eq41} we get
\begin{equation}\label{eq425}
\begin{aligned}
I_3&\le \int\int_{E_{\delta}^C} \left[C_1|\alpha|\theta_1\left(\frac{d(x)^{-\theta_2}+t^{\theta_1}}{\delta^{-\theta_2}}\right)^{C_1\alpha-1}\frac{t^{\theta_1-1}}{\delta^{-\theta_2}}\right]^{\frac{q+\alpha}{q-1}} V^{-\frac{\alpha+1}{q-1}}\,dx dt\\
&\le C\int\int_{E_{\delta}^C} |\alpha|^{\frac{q+\alpha}{q-1}}\left[d(x)^{-\theta_2}+t^{\theta_1}\right]^{\frac{(C_1\alpha-1)(q+\alpha)}{q-1}}t^{\frac{(\theta_1-1)(q+\alpha)}{q-1}}\delta^{\frac{\theta_2C_1\alpha (q+\alpha)}{q-1}} V^{-\frac{\alpha+1}{q-1}}\,dx dt.
\end{aligned}
\end{equation}
Arguing as in \eqref{eq416}, we can say that there exist $\bar{C}>0$ such that
$$
\delta^{\frac{\theta_2C_1\alpha (q+\alpha)}{q-1}}\le \bar{C}\,.
$$
Hence \eqref{eq425}, for some constant $C>0$, reduces to
\begin{equation}\label{eq426}
I_3\le C|\alpha|^{\frac{q+\alpha}{q-1}} \int\int_{E_{\delta}^C} V^{-\frac{\alpha+1}{q-1}} t^{\frac{(\theta_1-1)(q+\alpha)}{q-1}}\left[\left(d(x)^{-\theta_2}+t^{\theta_1}\right)^{-\frac{1}{\theta_2}}\right]^{-\theta_2\frac{(C_1\alpha-1)(q+\alpha)}{q-1}} \,dx dt.
\end{equation}
We have the following
\begin{claim}{} If $f:(0,+\infty)\to[0,+\infty)$ is a non decreasing function and if (HP$1)-(i)$ holds then, for any $0<\varepsilon<\varepsilon_0$ and for any $\delta>0$ small enough, we can write
\begin{equation}\label{eq427}
\begin{aligned}
\int\int_{E_{\delta}^C} f\left(\left[\left(d(x)^{-\theta_2}+t^{\theta_1}\right)^{-\frac{1}{\theta_2}}\right]\right) &t^{(\theta_1-1)\left(\frac{q}{q-1}-\varepsilon\right)}V^{-\frac{1}{q-1}+\varepsilon}\,dx dt \\
& \le C \int_0^{2\delta} f(z) z^{-\frac{q}{q-1}\theta_2-C_0\varepsilon-1}|\log z|^{s_2}\,dz,
\end{aligned}
\end{equation}
for some constant $C>0$.
\end{claim}

Inequality \eqref{eq427} can be proven similarly to \eqref{eq420} where one uses (HP$1)-(i)$ instead of (HP$1)-(ii)$.
We now apply \eqref{eq427} with $\varepsilon=\frac{|\alpha|}{q-1}>0$ to inequality \eqref{eq426}. We get
\begin{equation}\label{eq428}
I_3\le C|\alpha|^{\frac{q+\alpha}{q-1}}\int_0^{2\delta}z^{-\theta_2(C_1\alpha-1)\frac{q+\alpha}{q-1}-\frac{q}{q-1}\theta_2+\frac{C_0\alpha}{q-1}-1}|\log z|^{s_2}\,dz.
\end{equation}
We define
\begin{equation}\label{eq429}
\beta:=\frac{1}{q-1}\left(-\theta_2C_1\alpha (q+\alpha)+\theta_2\alpha+C_0\alpha\right),
\end{equation}
and due to \eqref{eq41b}, we have
$$
\beta \ge \frac{|\alpha|}{q-1} > 0.
$$
By plugging \eqref{eq429} into inequality \eqref{eq428} and using the change of variables $y=\beta\log z$, we get
\begin{equation}\label{eq430}
\begin{aligned}
I_3&\le C|\alpha|^{\frac{q+\alpha}{q-1}}\int_{-\infty}^0 e^y \left|\frac{y}{\beta}\right|^{s_2} \frac{1}{\beta}\,dy\\
&\le C\,|\alpha|^{\frac{q+\alpha}{q-1}}\,\beta^{-s_2-1}\\
&\le C\,|\alpha|^{\frac{1}{q-1}-s_2}.
\end{aligned}
\end{equation}
with $s_2$ as in (HP$1)-(i)$.

For any $n\in\mathbb{N}$ and $\delta>0$ small enough, due to inequalities \eqref{eq412}, \eqref{eq417}, \eqref{eq424} and \eqref{eq430}, inequality \eqref{eq49b} reduces to
\begin{equation}\label{eq431}
\begin{aligned}
\int_{0}^{\infty}\int_{\Omega} V\,u^{q+\alpha}\,\varphi_n^s\,dx dt &\le C |\alpha|^{-\frac{(p-1)q}{q-p+1}} \left[|\alpha|^{\frac{pq}{q-p+1}-s_4-1}+n^{-\frac{|\alpha|}{q-p+1}}\left|\log\left(\frac{\delta}{n}\right)\right|^{s_4}\right] \\
&+ C\left[ |\alpha|^{\frac{1}{q-1}-s_2}+n^{-\frac{|\alpha|}{q-1}}\left|\log\left(\frac{\delta}{n}\right)\right|^{s_2}\right],
\end{aligned}
\end{equation}
where $C>0$ does not depend on $\delta$ and $n$. By taking the limit in \eqref{eq431} as $n\to \infty$ for fixed small enough $\delta>0$, we get
\begin{equation}\label{eq432}
\begin{aligned}
0\le \int\int_{E_{\delta}} V\,u^{q+\alpha}\,dx dt&\le \int_{0}^{\infty}\int_{\Omega} V\,u^{q+\alpha}\,\varphi_n^s\,dx dt\\
&\le C\left[ |\alpha|^{\frac{p-1}{q-p+1}-s_4} + |\alpha|^{\frac{1}{q-1}-s_2}\right].
\end{aligned}
\end{equation}
Observe that, due to the definitions of $s_2$ in (HP$1)-(i)$ and $s_4$ in (HP$2)-(ii)$
$$
\frac{1}{q-1}-s_2>0\,,\quad\frac{p-1}{q-p+1}-s_4>0\,.
$$
Hence we can take the limit in \eqref{eq432} as $\delta\to 0$, and thus $\alpha \to 0^-$, obtaining by Fatou's Lemma
$$
\int_{0}^{\infty}\int_{\Omega} V\,u^{q}\,dx dt=0,
$$
which concludes the proof.
\end{proof}

As a consequence of Theorem \ref{teorema1} we prove Corollary \ref{corollario1}.
\begin{proof}[Proof of Corollary \ref{corollario1}]
We show that under the assumptions of Corollary \ref{corollario1}, hypothesis (HP$1$) is satisfied. Let us define
$$
\hat{E}_{\delta}:=E_{\frac{\delta}{2}}\setminus E_{\delta}
$$
and observe that
$$
\hat{E}_{\delta}\subset \left\{d(x)\ge \frac{\delta}{2}\right\}\times\bigg[0,\,\left(\frac{\delta}{2}\right)^{-\frac{\theta_2}{\theta_1}}\bigg]=: \Omega_{\frac{\delta}{2}}\times\bigg[0,\,\left(\frac{\delta}{2}\right)^{-\frac{\theta_2}{\theta_1}}\bigg],
$$
where $d(x)$ has been defined in \eqref{eq20}. Observe that for $\delta>0$ small enough
\begin{equation}\label{eq61}
\begin{aligned}
\int\int_{\hat{E}_{\delta}}& t^{(\theta_1-1)\left(\frac{q}{q-1}-\varepsilon\right)}\,V^{-\frac{1}{q-1}+\varepsilon}\,dx dt \\
&\le \int\int_{\hat{E}_{\delta}} t^{(\theta_1-1)\left(\frac{q}{q-1}-\varepsilon\right)} \left[g(t)h(x)\right]^{-\frac{1}{q-1}+\varepsilon}\,dx dt \\
& \le C\int_{\Omega_\frac{\delta}{2}}h(x)^{-\frac{1}{q-1}+\varepsilon}\,dx\,\int_0^{\left(\frac{\delta}{2}\right)^{-\frac{\theta_2}{\theta_1}}} t^{(\theta_1-1)\left(\frac{q}{q-1}-\varepsilon\right)} g(t)^{-\frac{1}{q-1}+\varepsilon}dt \\
&\le C\int_{\Omega_\frac{\delta}{2}}\left[d(x)^{-\sigma_1}\left(\log(1+d(x)^{-1})\right)^{-\delta_1}\right]^{-\frac{1}{q-1}+\varepsilon}dx\, \\
&\times \int_0^{\left(\frac{\delta}{2}\right)^{-\frac{\theta_2}{\theta_1}}} g(t)^{-\frac{1}{q-1}} (1+t)^{\alpha \varepsilon} t^{(\theta_1-1)\left(\frac{q}{q-1}-\varepsilon\right)}dt \\
& \le C\int_{\Omega_\frac{\delta}{2}} d(x)^{\frac{\sigma_1}{q-1}-\varepsilon \sigma_1}\left(\log(1+d(x)^{-1})\right)^{\frac{\delta_1}{q-1}-\varepsilon\delta_1}dx \\
&\times\left[\delta^{-\frac{\theta_2}{\theta_1}\left[(\theta_1-1)\left(\frac{q}{q-1}-\varepsilon\right)+\alpha\varepsilon\right]}\int_0^{\left(\frac{\delta}{2}\right)^{-\frac{\theta_2}{\theta_1}}} g(t)^{-\frac{1}{q-1}}dt\right] \\
& \le C\left|\log(\delta)\right|^{\frac{\delta_1}{q-1}-\varepsilon\delta_1} \left[\delta^{-\frac{\theta_2}{\theta_1}\left[(\theta_1-1)\left(\frac{q}{q-1}-\varepsilon\right)+\alpha\varepsilon\right]}\right]\delta^{-\frac{\theta_2}{\theta_1}\sigma_2}\left|\log(\delta)\right|^{\delta_2}\\
&\le C \delta^{-\frac{\theta_2}{\theta_1}\left[(\theta_1-1)\left(\frac{q}{q-1}-\varepsilon\right)+\alpha\varepsilon+\sigma_2\right]}\left|\log(\delta)\right|^{\frac{\delta_1}{q-1}-\varepsilon\delta_1+\delta_2},
\end{aligned}
\end{equation}
for $\theta_1, \theta_2\ge1$. For $C_0>0$ large and every $\varepsilon>0$ small enough, condition \eqref{eq23} of (HP$1$) is satisfied because
\begin{equation}\label{eq62}
\frac{\theta_2}{\theta_1}\left[\frac{q}{q-1}-\sigma_2\right]\,\ge\,0\quad \text{and}\quad \delta_2+\frac{\delta_1}{q-1} < \bar{s_2}\,.
\end{equation}
On the other hand, for $\varepsilon,\delta>0$ sufficiently small
\begin{equation}\label{eq63}
\begin{aligned}
\int\int_{\hat{E}_{\delta}}& d(x)^{-(\theta_2+1)p\left(\frac{q}{q-p+1}-\varepsilon\right)}\,V^{-\frac{p-1}{q-p+1}+\varepsilon}\,dx dt \\
&\le \int\int_{\hat{E}_{\delta}}d(x)^{-(\theta_2+1)p\left(\frac{q}{q-p+1}-\varepsilon\right)}\left[g(t)h(x)\right]^{-\frac{p-1}{q-p+1}+\varepsilon}\,dx dt \\
&\le \int_{\Omega_\frac{\delta}{2}}d(x)^{-(\theta_2+1)p\left(\frac{q}{q-p+1}-\varepsilon\right)}h(x)^{-\frac{p-1}{q-p+1}+\varepsilon}dx \int_0^{\left(\frac{\delta}{2}\right)^{-\frac{\theta_2}{\theta_1}}} g(t)^{-\frac{p-1}{q-p+1}+\varepsilon}dt\\
&\le C\int_{\Omega_\frac{\delta}{2}}d(x)^{-(\theta_2+1)p\left(\frac{q}{q-p+1}-\varepsilon\right)}\left[d(x)^{\sigma_1}\left(\log(1+d(x)^{-1})\right)^{\delta_1}\right]^{\frac{p-1}{q-p+1}-\varepsilon}dx \\
& \times\left[ \delta^{-\frac{\theta_2}{\theta_1}\alpha\varepsilon}\int_0^{\left(\frac{\delta}{2}\right)^{-\frac{\theta_2}{\theta_1}}} g(t)^{-\frac{p-1}{q-p+1}}dt\right]\\
& \le C \int_{\Omega_\frac{\delta}{2}}d(x)^{-(\theta_2+1)p\left(\frac{q}{q-p+1}-\varepsilon\right)+\sigma_1\frac{p-1}{q-p+1}-\varepsilon\sigma_1}\left(\log(1+d(x)^{-1})\right)^{\delta_1\frac{p-1}{q-p+1}-\varepsilon\delta_1}dx\\
&\times \left[\delta^{-\frac{\theta_2}{\theta_1}\alpha\varepsilon} \delta^{-\frac{\theta_2}{\theta_1}\sigma_4} \right] \\
&\le C  \delta^{-\frac{\theta_2}{\theta_1}(\alpha\varepsilon+\sigma_4)} |\log(\delta)|^{\delta_1\left(\frac{p-1}{q-p+1}-\varepsilon\right)} \int_{\Omega_\frac{\delta}{2}}d(x)^{-(\theta_2+1)p\left(\frac{q}{q-p+1}-\varepsilon\right)+\sigma_1\frac{p-1}{q-p+1}-\varepsilon\sigma_1}dx
\end{aligned}
\end{equation}
We define
$$
\beta:=-(\theta_2+1)p\left(\frac{q}{q-p+1}-\varepsilon\right)+\sigma_1\frac{p-1}{q-p+1}-\varepsilon\sigma_1
$$
and we observe that $\beta<-1$ for $\theta_2$ sufficiently large. Therefore, due to the boundedness of $\Omega_{\delta}$, inequality \eqref{eq63} reduces to
\begin{equation}\label{eq64}
\int\int_{\overline{E}_{\delta}} d(x)^{-(\theta_2+1)p\left(\frac{q}{q-p+1}-\varepsilon\right)}\,V^{-\frac{p-1}{q-p+1}+\varepsilon}\,dx dt
\le  C  \delta^{-\frac{\theta_2}{\theta_1}(\alpha\varepsilon+\sigma_4)+\beta+1} |\log(\delta)|^{\delta_1\left(\frac{p-1}{q-p+1}-\varepsilon\right)}
\end{equation}
For $\varepsilon,\delta>0$ small enough and for $\theta_2/\theta_1>0$ small enough, condition \eqref{eq24} is satisfied for some large $C_0>0$ because the hypotheses of the Corollary \ref{corollario1} guarantee that
$$
\sigma_1-\frac{\theta_2}{\theta_1}\sigma_4\frac{q-p+1}{p-1}\ge q+1\qquad \text{and}\qquad \delta_1\frac{p-1}{q-p+1}<\bar s_4.
$$
Thus (HP$1$) holds and we can apply Theorem \ref{teorema1} to obtain the result.
\end{proof}

\section{Proof of Theorem \ref{teorema2}}\label{proofnonlin2}

\begin{proof}[Proof of Theorem \ref{teorema2}]
Let us recall the family of functions $\varphi_n$ defined in \eqref{eq43}. We claim that $u^q\in L^1(\Omega\times(0,+\infty),Vd\mu dt)$. To prove this, we start by showing that for some constants $A>0$, $B>0$, $s\ge 1$, for every $\delta>0$ small enough and every $n\in \mathbb{N}$ we have
\begin{equation}\label{eq51}
\int_0^{\infty} \int_{\Omega}\varphi_n^su^qV\,dx dt \le A\left(\int_0^{\infty} \int_{\Omega}\varphi_n^su^qV\,dx dt\right)^{\frac{p-1}{pq}}\,+\,B.
\end{equation}
In order to prove \eqref{eq51} we apply Corollary \ref{corollario2} with $\varphi$ replaced by the family of functions $\varphi_n$. Let
\begin{equation}\label{eq51b}
C_1>\max\left\{ \frac{2(1+C_0+\theta_2)}{\theta_2q},\,\frac{2(\theta_2(q-1)+C_0+1)}{\theta_2(q-1)q},\, \frac{2C_0+1}{\theta_2(q-p+1)},\frac{2C_0+1}{\theta_2}\right\}\,,
\end{equation}
with $C_0>0$ and $\theta_2\ge 1$ as in (HP$2$). Then for any fixed $s\ge \max \left\{ 1, \frac{q+1}{q-1}, \frac{2pq}{q-p+1}\right\}$, $\delta>0$ sufficiently small, $\alpha=\frac{1}{\log\delta}<0$ and for every $n\in\mathbb{N}$, we have
\begin{equation}\label{eq52}
\begin{aligned}
\int_0^{\infty} & \int_{\Omega} V\,u^{q}\,\varphi^s\,dx \,dt \\
&\le C \left[ |\alpha|^{-1} \left(|\alpha|^{-\frac{(p-1)q}{q-p+1}} \int_0^{\infty}\int_{\Omega}V^{-\frac{p+\alpha-1}{q-p+1}}|\nabla \varphi_n |^{\frac{p(q+\alpha)}{q-p+1}}\, dx dt \right.\right. \\
& \left.\left.+  \int_0^{\infty}\int_{\Omega} V^{-\frac{\alpha+1}{q-1}}\,|\partial_t \varphi_n|^{\frac{q+\alpha}{q-1}}  \, dx dt \right)\right]^{\frac{p-1}{p}} \times \left( \int\int_{E_{\delta}^C} V\,u^{q}\varphi_n^s\,dx\,dt\right)^{\frac{(1-\alpha)(p-1)}{pq}}\\
& \times\left( \int\int_{E_{\delta}^C} V^{-\frac{(1-\alpha)(p-1)}{q-(1-\alpha)(p-1)}}\,|\nabla\varphi_n|^{\frac{pq}{q-(1-\alpha)(p-1)}}\,dx dt\right)^{\frac{q-(1-\alpha)(p-1)}{pq}} \\
&+C\,\left[|\alpha|^{-\frac{(p-1)q}{q-p+1}}\int_0^{\infty}\int_{\Omega} V^{-\frac{p+\alpha-1}{q-p+1}}\,|\nabla\varphi_n|^{\frac{p(q+\alpha)}{q-p+1}}\,dx dt \right.\\
&\left.+ \int_0^{\infty}\int_{\Omega}V^{-\frac{\alpha+1}{q-1}}\, |\partial_t \varphi_n|^{\frac{q+\alpha}{q-1}}\,dx dt\right]^{\frac{1}{q+\alpha}} \\
&\times \left(\int_0^{\infty}\int_{\Omega}V^{-\frac{1}{q+\alpha-1}}\, |\partial_t \varphi_n|^{\frac{q+\alpha}{q+\alpha-1}}\,dx dt\right)^{\frac{q+\alpha-1}{q+\alpha}}.
\end{aligned}
\end{equation}
where $E_{\delta}$ has been defined in \eqref{eq21}. We also define
\begin{align}
&J_1:=\int_0^{\infty}\int_{\Omega}V^{-\frac{p+\alpha-1}{q-p+1}}|\nabla \varphi_n |^{\frac{p(q+\alpha)}{q-p+1}}\, dx dt ; \label{eq53}\\
&J_2:=\int_0^{\infty}\int_{\Omega} V^{-\frac{\alpha+1}{q-1}}\,|\partial_t \varphi_n|^{\frac{q+\alpha}{q-1}}  \, dx dt;
\label{eq54}\\
&J_3:= \int\int_{E_{\delta}^C} V^{-\frac{(1-\alpha)(p-1)}{q-(1-\alpha)(p-1)}}\,|\nabla\varphi_n|^{\frac{pq}{q-(1-\alpha)(p-1)}}\,dx dt;
\label{eq55}\\
&J_4:=\int\int_{E_{\delta}^C}V^{-\frac{1}{q+\alpha-1}}\, |\partial_t \varphi_n|^{\frac{q+\alpha}{q+\alpha-1}}\,dx dt.
\label{eq56}
\end{align}
By using \eqref{eq53}, \eqref{eq54}, \eqref{eq55} and \eqref{eq56}, inequality \eqref{eq52} reads
\begin{equation}\label{eq57}
\begin{aligned}
\int_0^{\infty} & \int_{\Omega} V\,u^{q}\,\varphi^s\,dx \,dt \\
&\le C \left[ |\alpha|^{-1-\frac{(p-1)q}{q-p+1}} J_1\right]^{\frac{p-1}{p}} \left( \iint_{E_{\delta}^C} V\,u^{q}\varphi_n^s\,dx\,dt\right)^{\frac{(1-\alpha)(p-1)}{pq}} J_3^{\frac{q-(1-\alpha)(p-1)}{pq}} \\
&+C \left[ |\alpha|^{-1} J_2\right]^{\frac{p-1}{p}} \left( \iint_{E_{\delta}^C} V\,u^{q}\varphi_n^s\,dx\,dt\right)^{\frac{(1-\alpha)(p-1)}{pq}} J_3^{\frac{q-(1-\alpha)(p-1)}{pq}} \\
&+C\,\left[|\alpha|^{-\frac{(p-1)q}{q-p+1}}J_1+J_2\right]^{\frac{1}{q+\alpha}}J_4^{\frac{q+\alpha-1}{q+\alpha}}\\
&\le C \left[ |\alpha|^{-\frac{(p-1)q}{q-p+1}} J_1\right]^{\frac{p-1}{p}} \left( \iint_{E_{\delta}^C} V\,u^{q} \varphi_n^s\,dx\,dt\right)^{\frac{(1-\alpha)(p-1)}{pq}} \\ &\times\left[|\alpha|^{-\frac{(p-1)q}{q-(1-\alpha)(p-1)}} J_3\right]^{\frac{q-(1-\alpha)(p-1)}{pq}} \\
&+C J_2^{\frac{p-1}{p}} \left( \iint_{E_{\delta}^C} V\,u^{q}\varphi_n^s\,dx\,dt\right)^{\frac{(1-\alpha)(p-1)}{pq}}\left[ |\alpha|^{-\frac{(p-1)q}{q-(1-\alpha)(p-1)}} J_3\right]^{\frac{q-(1-\alpha)(p-1)}{pq}} \\
&+C\,\left[|\alpha|^{-\frac{(p-1)q}{q-p+1}}J_1+J_2\right]^{\frac{1}{q+\alpha}}J_4^{\frac{q+\alpha-1}{q+\alpha}}.
\end{aligned}
\end{equation}
Let us prove that, for $\delta>0$ sufficiently small and $|\alpha|=-\frac{1}{\log\delta}>0$ sufficiently small
\begin{align}
&\limsup_{n\to\infty}\,\left(|\alpha|^{-\frac{(p-1)q}{q-p+1}} J_1 \right) \le C, \label{eq58}\\
&\limsup_{n\to\infty}\,\left(|\alpha|^{-\frac{(p-1)q}{q-(1-\alpha)(p-1)}} J_3 \right) \le C, \label{eq59} \\
&\limsup_{n\to\infty} \,J_2 \le C, \label{eq510}\\
&\limsup_{n\to\infty} \,J_4 \le C, \label{eq511}
\end{align}
for some $C>0$ independent of $\alpha$.

We start by proving \eqref{eq58}. Observe that
\begin{equation}\label{eq512}
J_1\le\,C(I_1+I_2),
\end{equation}
with $I_1$ and $I_2$ defined in \eqref{eq46} and \eqref{eq47}, respectively. Arguing as in the proof of Theorem \ref{teorema1}, using condition \eqref{eq27} in
(HP$2)-(ii)$ in place of condition \eqref{eq24} in (HP$1)-(ii)$, we obtain, similar to \eqref{eq412},
\begin{equation}\label{eq514}
I_2\le C\,n^{-\frac{|\alpha|}{q-p+1}}\left|\log\left(\frac{\delta}{n}\right)\right|^{\bar{s}_4}
\end{equation}
and, similar to \eqref{eq424},
\begin{equation}\label{eq515}
I_1\le C\,|\alpha|^{\frac{pq}{q-p+1}-\bar{s}_4-1} = C\,|\alpha|^{\frac{q(p-1)}{q-p+1}} .
\end{equation}
Combining \eqref{eq512}, \eqref{eq514} and \eqref{eq515}, for some $C>0$ and for every $n\in\mathbb{N}$, we have
\begin{equation}\label{eq516}
|\alpha|^{-\frac{q(p-1)}{q-p+1}}J_1\le C\left(1+|\alpha|^{-\frac{q(p-1)}{q-p+1}} n^{-\frac{|\alpha|}{q-p+1}}\left|\log\left(\frac{\delta}{n}\right)\right|^{\bar{s}_4}\right).
\end{equation}
We can compute the limit as $n\to\infty$ on both sides of \eqref{eq516}, thus we obtain \eqref{eq58}.

Now observe that
\begin{equation}\label{eq517}
J_2\le\,C(I_3+I_4),
\end{equation}
with $I_3$ and $I_4$ defined in \eqref{eq48} and \eqref{eq49}, respectively. Then arguing as in the proof of Theorem \ref{teorema1}, due to condition \eqref{eq25} in (HP$2)-(i)$ with $\varepsilon=-\frac{\alpha}{q-1}>0$ we deduce, similar to \eqref{eq417}, for some positive constant $C$
\begin{equation}\label{eq518}
I_4\le C\,n^{-\frac{|\alpha|}{q-1}}\left|\log\left(\frac{\delta}{n}\right)\right|^{\bar{s}_2},
\end{equation}
Moreover, similar to \eqref{eq430}, we have
\begin{equation}\label{eq519}
I_3 \le C|\alpha|^{\frac{1}{q-1}-\bar{s}_2}=C.
\end{equation}
Combining \eqref{eq517}, \eqref{eq518} and \eqref{eq519}, for some $C>0$, every $n\in\mathbb{N}$ and for small enough $|\alpha|>0$ we have
\begin{equation*}
J_2\le C\left(1+ n^{-\frac{|\alpha|}{q-1}}\left|\log\left(\frac{\delta}{n}\right)\right|^{\bar{s}_2}\right).
\end{equation*}
Letting $n\to\infty$ we obtain \eqref{eq510}.

We now proceed to estimate $J_4$. Observe that
\begin{equation}\label{eq520}
J_4\le C\left(I_5 + I_6\right),
\end{equation}
where
$$
\begin{aligned}
&I_5:=\int\int_{E_{\delta}^C}V^{-\frac{1}{q+\alpha-1}}\, |\partial_t \varphi|^{\frac{q+\alpha}{q+\alpha-1}}\,dx dt,
&\qquad I_6:=\int\int_{E_{\delta}^C}V^{-\frac{1}{q+\alpha-1}}\, \varphi^{\frac{q+\alpha}{q+\alpha-1}}|\partial_t \eta_n|^{\frac{q+\alpha}{q+\alpha-1}}\,dx dt.
\end{aligned}
$$
Due to \eqref{eq41} we have
\begin{equation}\label{eq521}
\begin{aligned}
I_5&\le C\int\int_{E_{\delta}^C}V^{-\frac{1}{q+\alpha-1}}\, |\alpha|^{\frac{q+\alpha}{q+\alpha-1}}\left[\frac{d(x)^{-\theta_2}+t^{\theta_1}}{\delta^{-\theta_2}}\right]^{\frac{(C_1\alpha-1)(q+\alpha)}{q+\alpha-1}}\left(\frac{t^{\theta_1-1}}{\delta^{-\theta_2}}\right)^{\frac{q+\alpha}{q+\alpha-1}}\,dx dt \\
&\le C\,|\alpha|^{\frac{q+\alpha}{q+\alpha-1}}\int\int_{E_{\delta}^C}V^{-\frac{1}{q+\alpha-1}}\left[d(x)^{-\theta_2}+t^{\theta_1}\right]^{\frac{(C_1\alpha-1)(q+\alpha)}{q+\alpha-1}}\delta^{\frac{\theta_2C_1\alpha(q+\alpha)}{q+\alpha-1}}t^{(\theta_1-1)\left(\frac{(q+\alpha)}{q+\alpha-1}\right)}\,dx dt\\
&\le C\,|\alpha|^{\frac{q+\alpha}{q+\alpha-1}}\int\int_{E_{\delta}^C}V^{-\frac{1}{q+\alpha-1}}\left[\left(d(x)^{-\theta_2}+t^{\theta_1}\right)^{-\frac{1}{\theta_2}}\right]^{-\theta_2(C_1\alpha-1)\left(\frac{q}{q-1}-\frac{\alpha}{(q+\alpha-1)(q-1)}\right)}\\ &\quad\quad \times t^{(\theta_1-1)\left(\frac{q}{q-1}-\frac{\alpha}{(q+\alpha-1)(q-1)}\right)}\,dx dt,
\end{aligned}
\end{equation}
where we have used that there exists a positive constant $\bar C$ such that
$$
\delta^{\theta_2C_1\alpha\left(\frac{q+\alpha}{q+\alpha-1}\right)}=e^{\theta_2C_1\alpha\left(\frac{q+\alpha}{q+\alpha-1}\right)\log\delta}=e^{\theta_2C_1\left(\frac{q+\alpha}{q+\alpha-1}\right)}\le \bar C\,.
$$
\begin{claim}{} If $f:(0,+\infty)\to[0,+\infty)$ is a non decreasing function and if (HP$2$)$-(i)$ holds then, for any $0<\varepsilon<\varepsilon_0$ and for any $\delta>0$ small enough, we can write
\begin{equation}\label{eq522}
\begin{aligned}
\int\int_{E_{\delta}^C} f\left(\left[\left(d(x)^{-\theta_2}+t^{\theta_1}\right)^{-\frac{1}{\theta_2}}\right]\right) &t^{(\theta_1-1)\left(\frac{q}{q-1}+\varepsilon\right)}V^{-\frac{1}{q-1}-\varepsilon}\,dx dt \\
& \le C \int_0^{2\delta} f(z) z^{-\bar{s}_1-C_0\varepsilon-1}|\log z|^{\bar{s}_2}\,dz,
\end{aligned}
\end{equation}
for some constant $C>0$ with $\bar{s}_1$ and $\bar{s}_2$ as in \eqref{eq22}.
\end{claim}

Inequality \eqref{eq522} can be proven similarly to \eqref{eq420}, where one uses condition \eqref{eq26} in (HP$2)-(i)$ instead of (HP$1)-(i)$.
By using the latter claim with $\varepsilon=\frac{|\alpha|}{(q+\alpha-1)(q-1)}>0$ we obtain
$$
I_5\le C\,|\alpha|^{\frac{q+\alpha}{q+\alpha-1}}\int_0^{2\delta} z^{-\theta_2(C_1\alpha-1)\left(\frac{q+\alpha}{q+\alpha-1}\right)-\bar{s}_1-C_0\varepsilon-1}|\log z|^{\bar{s}_2}\,dz.
$$
Then observe that, due to \eqref{eq51b}, for $|\alpha|>0$ small
$$
-\theta_2(C_1\alpha-1)\left(\frac{q+\alpha}{q+\alpha-1}\right)-\bar{s}_1-C_0\varepsilon\,\ge\, \frac{|\alpha|}{(q-1)^2}=:b
$$
Now we define
$$
y:=b\,\log z,
$$
then there exists $\bar C>0$ such that for $|\alpha|>0$ small
\begin{equation}\label{eq523}
\begin{aligned}
I_5&\le C\,|\alpha|^{\frac{q+\alpha}{q+\alpha-1}}\int_{-\infty}^0 e^y\left|\frac{y}{b}\right|^{\bar{s}_2}\frac{1}{b}\,dy\\
&\le C \,|\alpha|^{\frac{q+\alpha}{q+\alpha-1}}b^{-\bar{s}_2-1}\int_{-\infty}^0 e^{y}|y|^{\bar{s}_2}\,dy\\
&\le C\,|\alpha|^{\frac{q+\alpha}{q+\alpha-1}} \left(\frac{|\alpha|}{(q-1)^2}\right)^{-\bar{s}_2-1} \\
&\le C\,|\alpha|^{\frac{q+\alpha}{q+\alpha-1}-\frac{1}{q-1}-1}\,\,\le\,\, \bar{C}\,.
\end{aligned}
\end{equation}
On the other hand, due to \eqref{eq41} and condition \eqref{eq26} in (HP$2)-(i)$ with $\varepsilon=\frac{|\alpha|}{(q+\alpha-1)(q-1)}$, by using the definition of $\tilde{E}_{\delta,n}$ in \eqref{eq45b}, for every $n\in\mathbb{N}$ we have
\begin{equation}\label{eq524}
\begin{aligned}
I_6&\le C\int\int_{\tilde{E}_{\delta,n}}V^{-\frac{1}{q+\alpha-1}}\, \left[\left(\frac{\delta}{n}\right)^{\theta_2}t^{\theta_1-1}\right]^{\frac{q+\alpha}{q+\alpha-1}}n^{\theta_2\alpha C_1\left(\frac{q+\alpha}{q+\alpha-1}\right)}\,dx dt\\
&\le C\,n^{\theta_2(C_1\alpha-1)\left(\frac{q+\alpha}{q+\alpha-1}\right)}\delta^{\theta_2\left(\frac{q+\alpha}{q+\alpha-1}\right)}\int\int_{\tilde{E}_{\delta,n}}V^{-\frac{1}{q-1}-\varepsilon}t^{(\theta_1-1)\left(\frac{q}{q-1}+\varepsilon\right)} \,dx dt\\
&\le C\,n^{\theta_2(C_1\alpha-1)\left(\frac{q+\alpha}{q+\alpha-1}\right)}\delta^{\theta_2\left(\frac{q+\alpha}{q+\alpha-1}\right)}\left(\frac{\delta}{n}\right)^{-\bar{s}_1-C_0\varepsilon} \left|\log\left(\frac{\delta}{n}\right)\right|^{\bar{s}_2}\\
&\le C\,n^{-\frac{|\alpha|}{q+\alpha-1}\left[\theta_2C_1(q+\alpha)+\frac{\theta_2}{q-1}-\frac{C_0}{q-1}\right]} \delta^{\frac{|\alpha|}{(q+\alpha-1)(q-1)}[\theta_2-C_0]} \left|\log\left(\frac{\delta}{n}\right)\right|^{\bar{s}_2}
\end{aligned}
\end{equation}
Now observe that there exists a positive constant $\bar C$ such that
\begin{equation}\label{eq525}
\delta^{\frac{|\alpha|}{(q+\alpha-1)(q-1)}[\theta_2-C_0]}=e^{\frac{|\alpha|}{(q+\alpha-1)(q-1)}[\theta_2-C_0]\log\delta}=e^{\frac{C_0-\theta_2}{(q+\alpha-1)(q-1)}}\le \bar C\,,
\end{equation}
and due to \eqref{eq51b}
\begin{equation}\label{eq526}
-\frac{|\alpha|}{q+\alpha-1}\left[\theta_2C_1(q+\alpha)+\frac{\theta_2}{q-1}-\frac{C_0}{q-1}\right]\le-\frac{|\alpha|}{(q-1)^2}.
\end{equation}
Combining \eqref{eq525} and \eqref{eq526} with \eqref{eq524} we obtain
\begin{equation}\label{eq527}
I_6\,\le\, C\,n^{-\frac{|\alpha|}{(q-1)^2}}\left|\log\left(\frac{\delta}{n}\right)\right|^{\bar{s}_2}.
\end{equation}
We now substitute \eqref{eq523} and \eqref{eq527} into inequality \eqref{eq520} thus we have, for some $C>0$ and for every $n\in\mathbb{N}$
$$
J_4\le C\left[1+n^{-\frac{|\alpha|}{(q-1)^2}}\left|\log\left(\frac{\delta}{n}\right)\right|^{\bar{s}_2}\right]\,.
$$
Letting $n\to\infty$ we get \eqref{eq511}.

In order to estimate integral $J_3$ defined in \eqref{eq55}, we define, for sufficiently small $|\alpha|>0$, the positive constant $\lambda$
\begin{equation}\label{eq528}
\lambda:=\frac{|\alpha|q(p-1)}{(q-p+1)[q-(1-\alpha)(p-1)]}.
\end{equation}
Observe that, for sufficiently small $|\alpha|>0$
\begin{equation}\label{eq529}
\frac{|\alpha|q(p-1)}{(q-p+1)^2}\,<\,\lambda\,<\,\frac{2|\alpha|q(p-1)}{(q-p+1)^2},
\end{equation}
and
\begin{equation}\label{eq530}
\frac{pq}{q-(1-\alpha)(p-1)}=\frac{\bar{s}_3}{\theta_2}+\lambda\,p,
\end{equation}
where $\bar{s}_3$ has been defined in \eqref{eq22} and $\theta_2\ge 1$ as in (HP$2$).
Thus by the definition of $\varphi_n$ in \eqref{eq43} and by \eqref{eq528}, for sufficiently small $|\alpha|>0$ and for every $n\in\mathbb{N}$ we have
\begin{equation}\label{eq531}
\begin{aligned}
J_3&\le C\int\int_{E_{\delta}^C}V^{-\lambda-\bar{s}_4}|\nabla\varphi|^{\frac{\bar{s}_3}{\theta_2}+\lambda\,p}\,dx dt + C\int\int_{\tilde{E}_{\delta,n}}V^{-\lambda-\bar{s}_4}\left(\varphi|\nabla\eta_n|\right)^{\frac{\bar{s}_3}{\theta_2}+\lambda\,p}\,dx dt\\
&=:C(I_7+I_8)\,,
\end{aligned}
\end{equation}
where $\tilde{E}_{\delta,n}$ has been defined in \eqref{eq45b}. Due to the very definition of $\varphi$ and $\eta_n$ in \eqref{eq41} and \eqref{eq42} respectively, and by \eqref{eq530} we get
$$
\begin{aligned}
I_8&\le C\int\int_{\tilde{E}_{\delta,n}}V^{-\lambda-\bar{s}_4}n^{C_1\alpha\theta_2\left(\frac{\bar{s}_3}{\theta_2}+\lambda\,p\right)}\left(\frac{\delta}{n}\right)^{\theta_2\left(\frac{\bar{s}_3}{\theta_2}+\lambda\,p\right)}d(x)^{-(\theta_2+1)\left(\frac{\bar{s}_3}{\theta_2}+\lambda\,p\right)}\,dx dt\\
&\le C\,n^{(C_1\alpha-1)(\bar{s}_3+\lambda\,p\theta_2)}\delta^{\bar{s}_3+\lambda\,p\theta_2}\int\int_{\tilde{E}_{\delta,n}}V^{-\lambda-\bar{s}_4}d(x)^{-(\theta_2+1)p\left(\frac{q}{q-p+1}+\lambda\,\right)}\,dx dt
\end{aligned}
$$
Now we use condition \eqref{eq28} in (HP$2)-(ii)$ with $\varepsilon=\lambda$ and we obtain, for every $n\in\mathbb{N}$ and for sufficiently small $\delta>0$
$$
\begin{aligned}
I_8&\le C\,n^{(C_1\alpha-1)p\theta_2\left(\frac{q}{q-p+1}+\lambda\right)}\,\,\delta^{p\,\theta_2\left(\frac{q}{q-p+1}+\lambda\right)}\left(\frac{\delta}{n}\right)^{-\frac{pq}{q-p+1}\theta_2-C_0\lambda}\left|\log\left(\frac{\delta}{n}\right)\right|^{\bar{s}_4}\\
&\le C\,n^{C_1\alpha p\theta_2\left(\frac{q}{q-p+1}+\lambda\right)-\lambda\,p\theta_2+C_0\lambda}\,\,\delta^{p\,\theta_2\lambda-C_0\lambda}\left|\log\left(\frac{\delta}{n}\right)\right|^{\bar{s}_4}\,.
\end{aligned}
$$
Due to the definition of $\lambda$ in \eqref{eq528}, inequality \eqref{eq529} and the definition of $C_1$ in \eqref{eq51b}, for sufficiently small $|\alpha|>0$ we write
$$
\begin{aligned}
C_1&\alpha\, p\,\theta_2\left(\frac{q}{q-p+1}+\lambda\right)-\lambda\,p\theta_2+C_0\lambda\\
&=(C_1\alpha-1)\,p\,\theta_2\frac{|\alpha|q(p-1)}{(q-p+1)[q-(1-\alpha)(p-1)]}+C_1\alpha\frac{p\,q\,\theta_2}{q-p+1}+\frac{C_0|\alpha|q(p-1)}{(q-p+1)[q-(1-\alpha)(p-1)]}\\
&\le (C_1\alpha-1)\,p\,\theta_2\frac{|\alpha|q(p-1)}{(q-p+1)^2}+C_1\alpha\frac{p\,q\,\theta_2}{q-p+1}+\frac{2C_0|\alpha|q(p-1)}{(q-p+1)^2} \\
&= C_1\alpha\,\theta_2\left[\frac{|\alpha|q\,p^2}{(q-p+1)^2}-\frac{|\alpha|q\,p}{(q-p+1)^2}+\frac{q\,p}{q-p+1}\right]-\frac{|\alpha|q(p-1)}{(q-p+1)^2}[p\theta_2-2C_0] \\
&= -\frac{|\alpha|}{(q-p+1)^2}\left[C_1\theta_2\,p\,q(q+(p-1)(|\alpha|-1))+(p\,\theta_2-2C_0)q(p-1)\right] \\
&\le -\frac{|\alpha|q}{(q-p+1)^2}[C_1\theta_2\,p(q-p+1)-2C_0p]\\
&\le -\frac{|\alpha|q\,p}{(q-p+1)^2}\,.
\end{aligned}
$$
Moreover, since $\alpha=\frac{1}{\log\delta}<0$, there exists $\bar{C}$ such that
$$
\delta^{\lambda(p\theta_2-C_0)}=e^{\lambda(p\theta_2-C_0)\log\delta}<e^{\frac{-|\alpha|\,q\,(p-1)}{(q-p+1)^2}(p\,\theta_2-C_0)|\log\delta|}\le \bar C
$$
Therefore we obtain the following bound on $I_8$
\begin{equation}\label{eq532}
I_8\,\le \, C \,n^{-\frac{|\alpha|\,p\,q}{(q-p+1)^2}}\left|\log\left(\frac{\delta}{n}\right)\right|^{\bar{s}_4}\,.
\end{equation}
On the other hand, by using the definition of $\varphi$ in \eqref{eq41} we can write
$$
I_7\le C\,|\alpha|^{\frac{pq}{q-(1-\alpha)(p-1)}}\int\int_{E_{\delta}^C}V^{-\lambda-\bar{s}_4}\left[\left(\frac{d(x)^{-\theta_2}+t^{\theta_1}}{\delta^{-\theta_2}}\right)^{C_1\alpha-1}\delta^{\theta_2}d(x)^{-(\theta_2+1)}\right]^{\frac{\bar{s}_3}{\theta_2}+\lambda\,p}\,dx dt\,,
$$
and we observe that there exists $\bar{C}>0$ such that for $|\alpha|>0$ small
$$
\delta^{C_1\alpha\,\theta_2\left( \frac{\bar{s}_3}{\theta_2}+\lambda\,p\right)}= \delta^{C_1\alpha\,\theta_2\left( \frac{p\,q}{q-(1-\alpha)(p-1)}\right)}<\delta^{C_1\alpha\,\theta_2\left(\frac{2pq}{q-p+1}\right)}=e^{C_1\alpha\,\theta_2\left(\frac{2pq}{q-p+1}\right)\log\delta}\le\,\bar C\,.
$$
Therefore we get
$$
I_7\le C\,|\alpha|^{\frac{pq}{q-(1-\alpha)(p-1)}}\int\int_{E_{\delta}^C}V^{-\lambda-\bar{s}_4}\left[d(x)^{-\theta_2}+t^{\theta_1}\right]^{(C_1\alpha-1)\left(\frac{\bar{s}_3}{\theta_2}+\lambda p\right)}d(x)^{-(\theta_2+1)\left(\frac{\bar{s}_3}{\theta_2}+\lambda\,p\right)}\,dx dt\,,
$$
We now state the following
\begin{claim}{} Let $f:(0,+\infty)\to[0,+\infty)$ be a non decreasing function and suppose that (HP$2$)$-(ii)$ holds. Then, for any $0<\varepsilon<\varepsilon_0$ and for any $\delta>0$ small enough, we can write
\begin{equation}\label{eq533}
\begin{aligned}
\int\int_{E_{\delta}^C} f\left(\left[\left(d(x)^{-\theta_2}+t^{\theta_1}\right)^{-\frac{1}{\theta_2}}\right]\right) &d(x)^{-(\theta_2+1)p\left(\frac{q}{q-p+1}+\varepsilon\right)}V^{-\frac{p-1}{q-p+1}-\varepsilon}\,dx dt \\
& \le C \int_0^{2\delta} f(z) z^{-\bar{s}_3-C_0\varepsilon-1}|\log z|^{\bar{s}_4}\,dz,
\end{aligned}
\end{equation}
for some constant $C>0$ with $\bar{s}_3$ and $\bar{s}_4$ as in \eqref{eq22}.
\end{claim}

Inequality \eqref{eq533} can be proven similarly to \eqref{eq420}, where one uses condition \eqref{eq28} in (HP$2)-(ii)$ instead of (HP$1)-(ii)$.
By using the latter claim with $\varepsilon=\lambda$ we get
\begin{equation}\label{eq534}
I_7\le C\,|\alpha|^{\frac{pq}{q-(1-\alpha)(p-1)}}\int_0^{2\delta} z^{-\theta_2(C_1\alpha-1)\left(\frac{\bar{s}_3}{\theta_2}+\lambda\,p\right)-\bar{s}_3-C_0\,\lambda-1}|\log z|^{\bar{s}_4}\,dz
\end{equation}
Observe that, since $\alpha<0$ and due to \eqref{eq51b}
$$
\begin{aligned}
-\theta_2&(C_1\alpha-1)\left(\frac{\bar{s}_3}{\theta_2}+\lambda\,p\right)-\bar{s}_3-C_0\,\lambda\\
&=-\theta_2C_1\alpha\frac{p\,q}{q-(1-\alpha)(p-1)}+p\,\theta_2\frac{|\alpha|\,q(p-1)}{(q-p+1)[q-(1-\alpha)(p-1)]}-C_0\frac{|\alpha|\,q(p-1)}{(q-p+1)[q-(1-\alpha)(p-1)]}\\
&\ge |\alpha|\theta_2\,C_1\frac{p\,q}{(q-p+1)^2}+p\,\theta_2\frac{|\alpha|\,q(p-1)}{(q-p+1)^2}-C_0\frac{2\,|\alpha|\,q(p-1)}{(q-p+1)^2} \\
&\ge \frac{|\alpha|\,q(p-1)}{(q-p+1)^2}\{\theta_2\,C_1-2C_0\}\\
&\ge\frac{|\alpha|\,q(p-1)}{(q-p+1)^2} \,=:\,a\,.
\end{aligned}
$$
We now set $y:=a\log z$ then, by using the definition of $\bar{s}_4$ in \eqref{eq22}, from \eqref{eq534} we deduce
\begin{equation}\label{eq535}
I_7\le C\,|\alpha|^{\frac{pq}{q-(1-\alpha)(p-1)}}a^{-\bar{s}_4-1}\int_{-\infty}^0 e^y\,|y|^{\bar{s}_4}\,dy\,\le\, C\,|\alpha|^{\frac{pq}{q-(1-\alpha)(p-1)}-\frac{q}{q-p+1}}.
\end{equation}
Combining together \eqref{eq531}, \eqref{eq532} and \eqref{eq535}, for any $\delta>0$ small enough and for every $n\in\mathbb{N}$ we have
$$
|\alpha|^{-\frac{q(p-1)}{q-(1-\alpha)(p-1)}}J_3\,\le\,C\,|\alpha|^{-\frac{q(p-1)}{q-(1-\alpha)(p-1)}}\left[|\alpha|^{\frac{pq}{q-(1-\alpha)(p-1)}-\frac{q}{q-p+1}}+ n^{-\frac{|\alpha|\,p\,q}{(q-p+1)^2}}\left|\log\left(\frac{\delta}{n}\right)\right|^{\bar{s}_4}\right].
$$
Then letting $n\to\infty$, for every $\delta>0$ small enough we obtain \eqref{eq59}. Now using \eqref{eq58}, \eqref{eq59}, \eqref{eq510} and \eqref{eq511} in \eqref{eq57}, for any $\delta>0$ sufficiently small and for every $n\in\mathbb{N}$ we get
$$
\begin{aligned}
\int_{0}^{\infty}\int_{\Omega}\varphi_n^s\,u^q\,V\,d\mu dt\,
&\le\,C'\,\left(\int\int_{E_{\delta}^C}\varphi_n^s\,u^q\,V\,dx dt\right)^{\frac{(1-\alpha)(p-1)}{p\,q}}\,+\,C''\\
&\le\,C'\,\left(\int_0^{\infty}\int_{\Omega}\varphi_n^s\,u^q\,V\,dx dt\right)^{\frac{(1-\alpha)(p-1)}{p\,q}}\,+\,C''\\
&\le\,C'\left(1+\int_0^{\infty}\int_{\Omega}\varphi_n^s\,u^q\,V\,dx dt\right)^{\frac{p-1}{p}}\,+\,C'' \\
&\le\,A\left(\int_0^{\infty}\int_{\Omega}\varphi_n^s\,u^q\,V\,dx dt\right)^{\frac{p-1}{p}}\,+\,B\,,
\end{aligned}
$$
where $A$, $B$ are positive constants independent of $n$, $\delta$ and $\frac{p-1}{p}\in (0,1)$. This easily implies that there exists $C>0$ such that, for sufficiently small $\delta>0$ and for every $n\in\mathbb{N}$
\begin{equation}\label{eq536}
\int_{0}^{\infty}\int_{\Omega}\varphi_n^s\,u^q\,V\,dx dt\,\le C\,.
\end{equation}
We then have
$$
\int\int_{E_\delta}u^q\,V\,dx dt\le\int_{0}^{\infty}\int_{\Omega}\varphi_n^s\,u^q\,V\,dx dt\,\le C\,.
$$
Thus letting $\delta\to 0$ we obtain that
\begin{equation}\label{eq537}
u^q\in L^1\left(\Omega\times(0,\infty);\,V\,dx dt\right)
\end{equation}
Now, we want to show that
$$
\int_0^{\infty}\int_\Omega u^q\,V\,dx dt=0.
$$
We use Lemma \ref{lemma3} where $\varphi$ is replaced by $\varphi_n$
\begin{equation}\label{eq537b}
\begin{aligned}
\int\int_{E_{\delta}}u^q\,V\,dx dt\,&\le\, \int_0^{\infty}\int_{\Omega}\varphi_n^su^q\,V\,dx dt \\
&\le C\,\left[|\alpha|^{-1-\frac{q(p-1)}{q-p+1}}\int_0^{\infty}\int_{\Omega}V^{-\frac{p+\alpha-1}{q-p+1}}|\nabla\varphi_n|^{\frac{p(q+\alpha)}{q-p+1}}\,dx dt \right. \\
&+\left. |\alpha|^{-1}\int_0^{\infty}\int_{\Omega}V^{-\frac{\alpha+1}{q-1}}|\partial_t\varphi_n|^{\frac{q+\alpha}{q-1}}\,dx dt\right]^{\frac{p-1}{p}}\left(\int\int_{E_{\delta}^C}\varphi_n^su^qV\,dx dt\right)^{\frac{(1-\alpha)(p-1)}{p\,q}}\\
&\times\left[\int\int_{E_{\delta}^C}V^{-\frac{(1-\alpha)(p-1)}{q-(1-\alpha)(p-1)}}|\nabla\varphi_n|^{\frac{pq}{q-(1-\alpha)(p-1)}}\,dx dt\right]^{\frac{q-(1-\alpha)(p-1)}{p\,q}} \\
&+\,C\,\left[\int\int_{E_{\delta}^C}\varphi_n^su^q\,V\,dx dt \right]^{\frac{1}{q}}\left[\int_0^{\infty}\int_{\Omega}V^{-\frac{1}{q-1}}|\partial_t\varphi_n|^{\frac{q}{q-1}}\,dxdt\right]^{\frac{q-1}{q}}\\
&\le\,C\,\left[|\alpha|^{-\frac{q(p-1)}{q-p+1}}J_1\right]^{\frac{p-1}{p}}\left(\int\int_{E_{\delta}^C}\varphi_n^su^qV\,dx dt\right)^{\frac{(1-\alpha)(p-1)}{p\,q}}\left[|\alpha|^{-\frac{q(p-1)}{q-(1-\alpha)(p-1)}}J_3\right]^{\frac{q-(1-\alpha)(p-1)}{p\,q}} \\
&+C\,J_2^{\frac{p-1}{p}}\left(\int\int_{E_{\delta}^C}\varphi_n^su^qV\,dx dt\right)^{\frac{(1-\alpha)(p-1)}{p\,q}}\left[|\alpha|^{-\frac{q(p-1)}{q-(1-\alpha)(p-1)}}J_3\right]^{\frac{q-(1-\alpha)(p-1)}{p\,q}}\\
&+\,C\,\left(\int\int_{E_{\delta}^C}\varphi_n^su^qV\,dx dt\right)^{\frac{1}{q}}\,J_5^{\frac{q-1}{q}}\,,
\end{aligned}
\end{equation}
where $J_1$, $J_2$, $J_3$ have been defined in \eqref{eq52}, \eqref{eq53}, \eqref{eq54} and
$$
J_5:=\int_0^{\infty}\int_{\Omega}V^{-\frac{1}{q-1}}|\partial_t\varphi_n|^{\frac{q}{q-1}}\,dx dt\,.
$$
Due to the definition of $\varphi_n$ in \eqref{eq43} we have
\begin{equation}\label{eq538}
\begin{aligned}
J_5&\le C\int_0^{\infty}\int_{\Omega}V^{-\frac{1}{q-1}}|\partial_t\varphi|^{\frac{q}{q-1}}\,dx dt\, + \int_0^{\infty}\int_{\Omega}V^{-\frac{1}{q-1}}\varphi^{\frac{q}{q-1}}|\partial_t\eta_n|^{\frac{q}{q-1}}\,dx dt\, \\
&:=C(I_9\,+\,I_{10})\,.
\end{aligned}
\end{equation}
By \eqref{eq41} we have
\begin{equation}\label{eq539}
I_9\le C\,|\alpha|^{\frac{q}{q-1}}\int\int_{E_{\delta}^C}V^{-\frac{1}{q-1}}\left[\left(d(x)^{-\theta_2}+t^{\theta_1}\right)^{-\frac{1}{\theta_2}}\right]^{-\theta_2(C_1\alpha-1)\frac{q}{q-1}}t^{(\theta_1-1)\frac{q}{q-1}}\,dx dt
\end{equation}
We now state the following
\begin{claim}{} Let $f:(0,+\infty)\to[0,+\infty)$ be a non decreasing function and suppose that (HP$2$)-$(i)$ holds. Then, for any $\delta>0$ small enough, we can write
\begin{equation}\label{eq540}
\begin{aligned}
\int\int_{E_{\delta}^C} f\left(\left[\left(d(x)^{-\theta_2}+t^{\theta_1}\right)^{-\frac{1}{\theta_2}}\right]\right) &t^{(\theta_1-1)\left(\frac{q}{q-1}\right)}V^{-\frac{1}{q-1}}\,dx dt \\
& \le C \int_0^{2\delta} f(z) z^{-\bar{s}_1-1}|\log z|^{\bar{s}_2}\,dz,
\end{aligned}
\end{equation}
for some constant $C>0$ with $\bar{s}_1$ and $\bar{s}_2$ as in \eqref{eq22}.
\end{claim}

Inequality \eqref{eq540} can be proven similarly to \eqref{eq420} where one uses the condition (HP$2)-(i)$ with $\varepsilon=0$ instead of (HP$1)-(ii)$.
We now use the latter claim in \eqref{eq539}, thus we have
\begin{equation}\label{eq541}
\begin{aligned}
I_9&\le C\,|\alpha|^{\frac{q}{q-1}}\int_0^{2\delta}z^{-\theta_2(C_1\alpha-1)\frac{q}{q-1}-\frac{q}{q-1}\theta_2-1}|\log z|^{\bar{s}_2}\,dz\\
&\le C\,|\alpha|^{\frac{q}{q-1}}\int_0^{2\delta}z^{-\theta_2C_1\alpha\frac{q}{q-1}-1}|\log z|^{\bar{s}_2}\,dz\\
&\le C\,|\alpha|^{\frac{q}{q-1}}\int_{-\infty}^0e^{y}\left|\frac{y}{\gamma}\right|^{\bar{s}_2}\frac{1}{\gamma}\,dy\\
&\le C\,|\alpha|^{\frac{q}{q-1}-\bar{s}_2-1} \\
&\le C
\end{aligned}
\end{equation}
where
$$
\gamma:=|\alpha|\,\theta_2\, C_1\frac{q}{q-1}\,\quad\quad\text{and}\quad\quad y:=\gamma\,\log z\,.
$$
On the other hand, by \eqref{eq42} we have
$$
\begin{aligned}
I_{10}&\le\, C\int\int_{\tilde{E}_{\delta,n}}V^{-\frac{1}{q-1}}\left[n^{\theta_2\,C_1\alpha}\left(\frac{\delta}{n}\right)^{\theta_2}t^{\theta_1-1}\right]^{\frac{q}{q-1}}\,dx dt \\
&\le\,C\,n^{\theta_2\,(C_1\alpha-1)\frac{q}{q-1}}\,\delta^{\theta_2\frac{q}{q-1}}\int\int_{\tilde{E}_{\delta,n}}V^{-\frac{1}{q-1}}t^{(\theta_1-1)\frac{q}{q-1}}dx dt
\end{aligned}
$$
Then, due to (HP$2)-(ii)$ with $\varepsilon=0$ we have
\begin{equation}\label{eq542}
\begin{aligned}
I_{10}&\le\,C\,n^{\theta_2\,(C_1\alpha-1)\frac{q}{q-1}+\frac{q}{q-1}\theta_2}\,\delta^{\theta_2\frac{q}{q-1}-\frac{q}{q-1}\theta_2}\left|\log\left(\frac{\delta}{n}\right)\right|^{\bar{s}_2}\\
&\le n^{-|\alpha|\theta_2\,C_1\frac{q}{q-1}}\left|\log\left(\frac{\delta}{n}\right)\right|^{\bar{s}_2}
\end{aligned}
\end{equation}
Now, combining \eqref{eq538}, \eqref{eq541} and \eqref{eq542} we get
$$
J_5\le\, C\left[1+n^{-|\alpha|\theta_2\,C_1\frac{q}{q-1}}\left|\log\left(\frac{\delta}{n}\right)\right|^{\bar{s}_2}\right]
$$
By letting $n\to\infty$ we obtain
\begin{equation}\label{eq543}
\limsup_{n\rightarrow+\infty}J_5\le \,C
\end{equation}
Finally we use inequalities \eqref{eq58}, \eqref{eq59}, \eqref{eq510} and \eqref{eq543} into \eqref{eq537b} and, passing to the $\limsup$ as $n\to\infty$, we obtain for some constant $C>0$
\begin{equation}\label{eq544}
\int\int_{E_{\delta}} u^q\,V\,dx dt\le C\left[\left(\int\int_{E_{\delta}^C}u^q\,V\,dx dt\right)^{\frac{(1-\alpha)(p-1)}{p}}+\left(\int\int_{E_{\delta}^C}u^q\,V\,dx dt\right)^{\frac{1}{q}}\right].
\end{equation}
Now we can pass to the limit in \eqref{eq544} as $\delta\to0$, and thus as $\alpha\to 0^-$, and conclude by using Fatou's Lemma and \eqref{eq537} that
$$
\int_0^{\infty}\int_{\Omega}u^q\,V\,dx dt=0.
$$
Thus $u=0$ a.e. in $\Omega\times[0,\infty)$.

\end{proof}

\section{Proof of Theorem \ref{teorema3}}\label{prooflin}

Throughout this section we always assume that $\partial\Omega$ is of class $C^3$. We now introduce two Lemmas that will be used in the proof of Theorem \ref{teorema3}. Let us first observe that, under the assumptions of Theorem \ref{teorema3}, the Green function $G(x,y)$ associated to the laplacian operator $-\Delta$ satisfies the following bound
\begin{equation}\label{eq65}
G(x,y)\le C\min\left\{1,\, \frac{d(x)d(y)}{|x-y|^2}\right\}\,|x-y|^{2-N},
\end{equation}
for some $C>0$ and $d(x)$ as in \eqref{eq20}. See \cite{Hu}, \cite{Z}; see also \cite{Da}, \cite{FMT}.

\begin{lemma}\label{lemma4}
Suppose that \eqref{eq65} holds and define
\begin{equation}\label{eq66}
\psi(x):=\int_{\Omega} G(x,y)\,d(y)^{\beta}\,dy,
\end{equation}
for $\beta>-1$. Then there exist $c=c(\beta)>0$ such that
\begin{equation}\label{eq67}
0\le\psi(x)\le c\,d(x) \quad \text{for every}\,\,\,x\in\Omega,
\end{equation}
\end{lemma}

\begin{proof}
Let us fix $x\in\Omega$ such that $d(x)>0$. Then, for any $y\in\Omega$
\begin{equation}\label{eq68}
d(y)\ge 2|x-y|,
\end{equation}
or
\begin{equation}\label{eq69}
d(y)\le 2|x-y|.
\end{equation}
Therefore we write
$$
\psi(x)=\int_{\{d(y)\ge2|x-y|\}}G(x,y)d(y)^{\beta}\,dy + \int_{\{d(y)\le2|x-y|\}}G(x,y)d(y)^{\beta}\,dy
$$
Moreover observe that, for any $z\in\partial\Omega$,
$$
|y-z|\le|x-z|+|y-x|.
$$
If we fix $z\in\partial\Omega$ such that $d(x)=|x-z|$ then the latter can be rewritten as
\begin{equation}\label{eq610}
|y-z|\le d(x)+|y-x|.
\end{equation}
Combining \eqref{eq68} and \eqref{eq610}, it follows that
\begin{equation}\label{eq611}
2|x-y|\le d(y)\le |y-z|\le d(x)+|y-x| \Longrightarrow |x-y|\le d(x).
\end{equation}
If $\beta>0$ we have
$$\int_\Omega G(x,y)d(y)^\beta\,dy\leq (\operatorname{diam}\Omega)^\beta\int_\Omega G(x,y)\,dy,$$
thus w.l.o.g. we can consider only the case when $-1<\beta\le 0$. Then, due to \eqref{eq65}, \eqref{eq68} and \eqref{eq611}
$$
\begin{aligned}
0\le &\int_{\{d(y)\ge2|x-y|\}}G(x,y)d(y)^{\beta}\,dy\\
&\le c \int_{\{d(y)\ge2|x-y|\}} \frac{d(y)^{\beta}}{|x-y|^{N-2}}\,dy\\
&\le c \int_{\{d(y)\ge2|x-y|\}} \frac{d(x)d(y)^{\beta}}{|x-y|^{N-1}}\,dy\\
&\le c \int_{\{d(y)\ge2|x-y|\}} \frac{d(x)}{|x-y|^{N-1-\beta}}\,dy.
\end{aligned}
$$
Now, since $-1<\beta\le 0$
\begin{equation}\label{eq612}
c \int_{\{d(y)\ge2|x-y|\}} \frac{d(x)}{|x-y|^{N-1-\beta}}\,dy\,\,\, \le\,\,\, c\,d(x)\int_{B_R(x)} \frac{1}{|x-y|^{N-1-\beta}}\,dy \le c\,d(x),
\end{equation}
where $R:=\diam(\Omega)$.
Similarly, due to \eqref{eq65} and \eqref{eq69}
\begin{equation}\label{eq613}
\begin{aligned}
0\le &\int_{\{d(y)\le2|x-y|\}}G(x,y)d(y)^{\beta}\,dy\\
&\le c \int_{\{d(y)\le2|x-y|\}} \frac{d(x)d(y)^{1+\beta}}{|x-y|^{N}}\,dy\\
&\le c \int_{\{d(y)\le2|x-y|\}} \frac{d(x)}{|x-y|^{N-(1+\beta)}}\,dy\\
&\le c\,d(x)\int_{B_R(x)}\frac{1}{|x-y|^{N-(1+\beta)}}\,dy \\
&\le c\,d(x)
\end{aligned}
\end{equation}
Finally, due to \eqref{eq612} and \eqref{eq613}, for any $x\in\Omega$, there exists $c=c(\beta)$ such that
$$
0\le \psi(x)\le c\,d(x).
$$

\end{proof}

\begin{lemma}\label{lemma5}
Suppose that \eqref{eq65} holds. Let us recall the definition of $\psi$ in \eqref{eq66} and suppose that
\begin{equation}\label{eq614}
\beta>-2.
\end{equation}
Then there exist $M>0$ such that
\begin{equation}\label{eq615}
0\le\psi(x)\le M \quad \text{for any}\,\,\,x\in\Omega,
\end{equation}
\end{lemma}

\begin{proof}
By Lemma \ref{lemma4} we only need to consider the case $-2<\beta\le-1$. For every $\varepsilon>0$ small enough, let $\Omega_\varepsilon$ be defined as in \eqref{eq616}. Moreover let $G_{\varepsilon}(x,y)$ be the Green function associated to the operator $-\Delta$ for $x,y\in \Omega_{\varepsilon}$. For every $\varepsilon>0$, let
\begin{equation}\label{eq617}
u_{\varepsilon}(x):=\int_{\Omega_{\varepsilon}} G_{\varepsilon}(x,y)\,d(y)^{\beta}\,dy.
\end{equation}
Observe that, for every $\varepsilon>0$, $u_{\varepsilon}\in C^{\infty}(\operatorname{Int}(\Omega_{\varepsilon}))\cap C^0(\Omega_{\varepsilon})$, $u_{\varepsilon}>0$ in $\operatorname{Int}(\Omega_{\varepsilon})$ and it solves the following problem
$$
\begin{cases}
-\Delta u_{\varepsilon}(x)=d(x)^{\beta} & \quad \text{in}\,\,\operatorname{Int}(\Omega_{\varepsilon})\\
\quad\quad\quad u_{\varepsilon}=0&\quad \text{on}\,\,\partial\Omega_{\varepsilon}\,.
\end{cases}
$$
Moreover, due to assumption \eqref{eq614}, see \cite{PPT}, there exists $v:\bar{\Omega} \to\mathbb{R}$, $v\in C^0(\bar\Omega)$, $v>0$ in $\Omega$ such that $v$ is a solution to problem
$$
\begin{cases}
-\Delta v(x)=d(x)^{\beta} & \quad \text{in}\,\,\Omega\\
\quad\quad\quad v=0&\quad \text{on}\,\,\partial\Omega\,.
\end{cases}
$$
Observe that, due to the maximum principle, it follows that
\begin{equation}\label{eq618}
0<u_{\varepsilon}<v\quad\text{in}\,\,\operatorname{Int}(\Omega_{\varepsilon})\quad\text{for any}\,\,\varepsilon>0.
\end{equation}
Moreover, for $0<\varepsilon_1<\varepsilon_2$ one has
\begin{equation}\label{eq619}
u_{\varepsilon_2}(x)\le u_{\varepsilon_1}(x)\quad \text{for any}\,\,\, x\in\Omega_{\varepsilon_2}.
\end{equation}
Hence, the family of functions $\{u_{\varepsilon}\}_{\varepsilon>0}$, due to \eqref{eq618} and \eqref{eq619}, admits a finite limit for $\varepsilon\to0$, in particular we write
\begin{equation}\label{eq620}
\lim_{\varepsilon\to0}\,u_{\varepsilon}(x)\,=w(x)\quad\text{for any}\,\,\,x\in\Omega,
\end{equation}
and $0<w(x)\le v(x)$ for any $x\in\Omega$.
Now observe that
$$
G_{\varepsilon}(x,y)\nearrow G(x,y)\quad \text{as}\,\,\,\varepsilon\to 0\,\,\,\text{for any}\,\,x,y\in\Omega\,.
$$
It follows by the Monotone Convergence Theorem that for any $\varepsilon>0$ one has
\begin{equation}\label{eq621}
u_{\varepsilon}(x)=\int_{\Omega}G_{\varepsilon}(x,y)d(y)^{\beta}\,dy\longrightarrow\int_{\Omega}G(x,y)d(y)^{\beta}\,dy\quad \text{as}\,\,\,\varepsilon\to 0.
\end{equation}
Hence, due to \eqref{eq618}, \eqref{eq620} and \eqref{eq621}, for any $x\in\Omega$ we can write
$$
w(x)=\int_{\Omega}G(x,y)d(y)^{\beta}\,dy,\quad\text{and}\quad 0\le \int_{\Omega}G(x,y)d(y)^{\beta}\,dy\le v(x)\,.
$$
Finally, since $v\in C^0(\overline\Omega)$, there exists $M>0$ such that
$$
0\le \int_{\Omega}G(x,y)d(y)^{\beta}\,dy\le M.
$$
\end{proof}
We are now ready to prove Theorem \ref{teorema3}.
\begin{proof}[Proof of Theorem \ref{teorema3}]
We want to construct a subsolution and a supersolution to problem \eqref{problema2} which will be denoted by $\underline u$ and $\overline{u}$ respectively. We  set
$$
\underline u\equiv 0.
$$
On the other hand, in order to construct $\overline u$, let us define, for any $\lambda>0$
\begin{equation}\label{eq622}
S_{\lambda}=\{v\in C^0(\bar{\Omega})\,:\, 0\le v(x)\le \lambda \,d(x),\,\,\forall x\in\Omega\}.
\end{equation}
with $d(x)$ as in \eqref{eq20}. Moreover we define the map $T:S_{\lambda}\to S_{\lambda}$
\begin{equation}\label{eq623}
Tv(x)=\lambda^q\int_{\Omega}G(x,y)\,dy+\int_{\Omega}G(x,y)V(y)v(y)^q\,dy.
\end{equation}
We prove that $T$ is well defined and that it is a contraction map for $\lambda>0$ small enough. Observe that,
by to Lemma \ref{lemma4} with $\beta=0$, one has for some $c_1>0$
\begin{equation}\label{eq624}
0\le \lambda^q\int_{\Omega}G(x,y)\,dy\le c_1\,\lambda^qd(x),\quad \text{for every}\,\,x\in\Omega.
\end{equation}
Similarly, due to \eqref{eq215}, Lemma \ref{lemma4} with $\beta=-\sigma_1+q$ and \eqref{eq216}, for some $c_2>0$\begin{equation}\label{eq625}
0\le \int_{\Omega}G(x,y)V(y)v(y)^q\,dy\le c\lambda^q\int_{\Omega}G(x,y)d(y)^{-\sigma_1+q}\,dy\le c_2\,\lambda^q\,d(x).
\end{equation}
By using \eqref{eq624} and \eqref{eq625}, \eqref{eq623} yields for some $C>0$ and $\lambda>0$ small enough
$$
0\le Tv(x)\le C\lambda^q\,d(x)\le\,\lambda\,d(x)\quad \text{for any}\,\,\,x\in\Omega.
$$
Hence, for a sufficiently small $\lambda>0$, the function $Tv:\bar{\Omega}\to\R$ is continuous and thus the map $T:S_{\lambda}\to S_{\lambda}$ is well defined. Let us now show that $T$ is a contraction map, for $\lambda>0$ small enough. Fix $w,v\in S_{\lambda}$, then for any $x\in\Omega$
$$
\begin{aligned}
|Tw(x)-Tv(x)|&\le \int_{\Omega}G(x,y)V(y)|w^q(y)-v^q(y)|\,dy \\
& \le\int_{\Omega}G(x,y)V(y)q\xi(y)^{q-1}|w(y)-v(y)|\,dy,
\end{aligned}
$$
for some $\xi(y)$ between $w(y)$ and $v(y)$. Then $0\leq\xi(y)\leq\lambda d(y)$ and hence, due to Lemma \ref{lemma5} with $\beta=-\sigma_1+q-1$ and \eqref{eq216},
$$
\begin{aligned}
|Tw(x)-Tv(x)|
&\le C\,\left(\int_{\Omega}G(x,y)d(y)^{-\sigma_1+q-1}\,dy\right)\,\lambda^{q-1}\|w-v\|_{L^{\infty}(\Omega)}\\
&\le C\, M\,\lambda^{q-1}\|w-v\|_{L^{\infty}(\Omega)}.
\end{aligned}
$$
Thus we have, for $\lambda>0$ small enough,
$$
\|Tw-Tv\|_{L^{\infty}(\Omega)}\le\frac 12\|w-v\|_{L^{\infty}(\Omega)},
$$
hence $T$ is a contraction map. Therefore, there exists $\varphi\in S_{\lambda}$ such that $\varphi=T\varphi$. In particular, we have
\begin{itemize}
\item[(i)] $0\le \varphi(x)\le \lambda\,d(x)\quad \text{for any}\,\,\, x\in\bar{\Omega}$;
\item[(ii)]$\varphi$ is a solution of
$$
\begin{cases}
-\Delta\varphi=\lambda^q+V\varphi^q&\quad \text{in}\,\,\Omega \\
\varphi=0 &\quad\text{on}\,\,\partial\Omega
\end{cases}
$$
\item[(iii)] $\varphi>0$ in $\Omega$.
\end{itemize}
We now set $\overline{u}(x,t)=\varphi(x)$ and show that $\overline{u}$ is a supersolution to problem \eqref{problema2}. Observe that
\begin{itemize}
\item[(i)] $\partial_t \overline{u}-\Delta\overline{u}=-\Delta\varphi=\lambda^q+V\,\varphi^q\ge V\,\overline{u}^q\quad$ in $\Omega\times (0,+\infty)$;
\item[(ii)] $\overline{u}(x,t)=\varphi(x)=0\quad $ for any $x\in\partial\Omega$, $t\in(0,+\infty)$;
\item[(iii)] $\overline{u}>0$ in $\Omega\times(0,+\infty)$;
\item[(iv)] $0\le u_0(x)\le \overline{u}(x,0)\,\,$ for any $x\in\Omega$, if $\varepsilon$ is small enough; indeed we can apply the Hopf's Lemma and if $n$ denotes the inward normal unit vector to $\partial\Omega$ deduce that
$$
\frac{\partial\varphi}{\partial n}(x)>0,\quad \text{for any}\,\,x\in\partial\Omega.
$$
Then, due to the compactness of $\bar{\Omega}$ and the continuity of $\varphi$ in $\Omega$, we observe that there exists $\alpha>0$ such that
$$
\varphi(x)\ge\alpha\,d(x)\quad \text{for any}\,\,x\in\bar{\Omega}.
$$
Now, if $\varepsilon>0$ in \eqref{eq214} is sufficiently small, we have that
$$
0\le u_0(x)\le\varepsilon \,d(x)\le\alpha\,d(x)\le \varphi(x)=\overline{u}(x,0)\quad\text{for any}\,\,x\in\bar{\Omega}.
$$
\end{itemize}
Thus $\overline{u}:\bar{\Omega}\times[0,+\infty)\to\R$ is a supersolution to problem \eqref{problema2}, such that $\overline{u}\geq \underline{u}$ in $\bar{\Omega}\times[0,+\infty)$. Finally, we conclude that there exists a solution $u:\Omega\times[0,+\infty)\to\R$ of problem \eqref{problema2} such that
$$
0\le u(x)\le \overline u(x)\quad \text{for any}\,\,x\in\bar{\Omega}.
$$
\end{proof}

\section{Proof of Theorem \ref{teorema4} and of Corollary \ref{corollario4}}\label{prooflin2}

We introduce some auxiliary Lemmas that are needed in the proof of Theorem \ref{teorema4}.
\begin{lemma}\label{lemma71}
Let $V\in L^1_{\textrm{loc}}(\Omega)$, with $V(x)>0$ a.e., and assume that the initial condition satisfies $u_0\in L^1_{loc}(\Omega)$, with $u_0\ge 0$ a.e.
Let $u\ge 0$ be a weak solution of problem \eqref{problema2}. If $\alpha>\frac{2q}{q-1}$ and $\psi\in C^{2,1}_{x,t}(\Omega\times[0,T))$, $\psi\ge 0$ a.e. in $\Omega\times[0,T)$ with compact support in $\Omega\times[0,T)$, then
\begin{equation}\label{eq71}
\begin{aligned}
\int_0^T\int_{\Omega}u^qV\psi^{\alpha}\,dx dt&\le 2^{\frac{1}{q-1}}\left\{\int_0^T\int_{\Omega}V^{-\frac{1}{q-1}}\psi^{\alpha-\frac{2q}{q-1}}\left|\alpha(\alpha-1)|\nabla\psi|^2+\alpha\psi\Delta\psi\right|^{\frac{q}{q-1}}\,dxdt\right.\\
&\left.+\int_0^T\int_{\Omega}V^{-\frac{1}{q-1}}\psi^{\alpha-\frac{2q}{q-1}}\left|\alpha\psi\psi_t\right|^{\frac{q}{q-1}}\,dxdt\right\}.
\end{aligned}
\end{equation}
\end{lemma}

\begin{proof}
Using the definition of weak solution to problem \eqref{problema2} and Young inequality with coefficients $q$ and $\frac{q}{q-1}$ we have
$$
\begin{aligned}
\int_0^T\int_{\Omega}u^qV\psi^{\alpha}\,dx dt&\le\int_0^T\int_{\Omega}u\left|(\psi^{\alpha})_t+\Delta(\psi^{\alpha})\right|\,dx dt-\int_{\Omega}u_0(x)\psi^{\alpha}(x,0)\,dx\\
&\le\frac{1}{q}\int_0^T\int_{\Omega}u^qV\psi^{\alpha}\,dx dt+\frac{q-1}{q}\int_0^T\int_{\Omega}(V\psi^{\alpha})^{-\frac{1}{q-1}}\left|(\psi^{\alpha})_t+\Delta(\psi^{\alpha})\right|^{\frac{q}{q-1}}\,dx dt
\end{aligned}
$$
Reordering terms we get
$$
\begin{aligned}
\int_0^T\int_{\Omega}u^qV\psi^{\alpha}\,dx dt&\le\int_0^T\int_{\Omega}(V\psi^{\alpha})^{-\frac{1}{q-1}}\left|\alpha\psi^{\alpha-1}\psi_t+\alpha(\alpha-1)\psi^{\alpha-2}|\nabla\psi|^2+\alpha\psi^{\alpha-1}\Delta\psi\right|^{\frac{q}{q-1}}\,dx dt\\
&\le\int_0^T\int_{\Omega}V^{-\frac{1}{q-1}}\psi^{-\frac{\alpha}{q-1}+\frac{q(\alpha-2)}{q-1}}\left|\alpha\psi\psi_t+\alpha(\alpha-1)|\nabla\psi|^2+\alpha\psi\Delta\psi\right|^{\frac{q}{q-1}}\,dx dt\\
&\le2^{\frac{1}{q-1}}\left\{\int_0^T\int_{\Omega}V^{-\frac{1}{q-1}}\psi^{\alpha-\frac{2q}{q-1}}\left|\alpha(\alpha-1)|\nabla\psi|^2+\alpha\psi\Delta\psi\right|^{\frac{q}{q-1}}\,dx dt\right.\\&\left.+\int_0^T\int_{\Omega}V^{-\frac{1}{q-1}}\psi^{\alpha-\frac{2q}{q-1}}\left|\alpha\psi\psi_t\right|^{\frac{q}{q-1}}\,dx dt\right\}
\end{aligned}
$$
This proves the thesis.
\end{proof}

\begin{lemma}\label{lemma72}
Let the assumptions of Lemma \ref{lemma71} hold. Moreover let $K\subset\Omega\times[0,T)$ be a compact set and let $\psi$ be such that $\psi\equiv 1$ in $K$. Let $S_k:= (\Omega\times[0,T))\setminus K$ then
\begin{equation}\label{eq72}
\begin{aligned}
\int_0^T\int_{\Omega}u^qV\psi^{\alpha}\,dx dt&\le 2^{\frac{1}{q}}\left(\int\int_{S_k}u^qV\psi^{\alpha}\,dxdt\right)^{\frac 1q}\\&\times \left\{\left[\int\int_{S_k}V^{-\frac{1}{q-1}}\psi^{\alpha-\frac{2q}{q-1}}\left|\alpha(\alpha-1)|\nabla\psi|^2+\alpha\psi\Delta\psi\right|^{\frac{q}{q-1}}\,dxdt\right]^{\frac{q-1}{q}}\right.\\&\left.+\left[\int\int_{S_k}V^{-\frac{1}{q-1}}\psi^{\alpha-\frac{2q}{q-1}}\left|\alpha\psi\psi_t\right|^{\frac{q}{q-1}}\,dxdt\right]^{\frac{q-1}{q}}\right\}.
\end{aligned}
\end{equation}
\end{lemma}

\begin{proof}
Similarly to the proof of Lemma \ref{lemma71}, using the definition of weak solution of problem \eqref{problema2} and H\"older inequality with coefficients $q$ and $\frac{q}{q-1}$ we get
$$
\begin{aligned}
\int_0^T\int_{\Omega}u^qV\psi^{\alpha}\,dx dt&\le \left(\int\int_{S_K}u^qV\psi^{\alpha}\,dx dt\right)^{\frac 1q}\left(\int\int_{S_K}V^{-\frac{1}{q-1}}\psi^{-\frac{\alpha}{q-1}}\left|(\psi^{\alpha})_t+\Delta(\psi^{\alpha})\right|^{\frac{q}{q-1}}\,dx dt\right)^{\frac{q-1}{q}} \\
&= \left(\int\int_{S_k}u^qV\psi^{\alpha}\,dx dt\right)^{\frac 1q} \\
&\times\left(\int\int_{S_k}V^{-\frac{1}{q-1}}\psi^{\alpha-\frac{2q}{q-1}}\left|\alpha\psi\psi_t+\alpha(\alpha-1)|\nabla\psi|^2+\alpha\psi\Delta\psi\right|^{\frac{q}{q-1}}\,dx dt\right)^{\frac{q-1}{q}}\\
&\le2^{\frac{1}{q}} \left(\int\int_{S_k}u^qV\psi^{\alpha}\,dx dt\right)^{\frac 1q}\\
&\times\left\{\left[\int\int_{S_k}V^{-\frac{1}{q-1}}\psi^{\alpha-\frac{2q}{q-1}}\left|\alpha(\alpha-1)|\nabla\psi|^2+\alpha\psi\Delta\psi\right|^{\frac{q}{q-1}}\,dx dt\right]^{\frac{q-1}{q}}\right.\\&\left.+\left[\int\int_{S_k}V^{-\frac{1}{q-1}}\psi^{\alpha-\frac{2q}{q-1}}\left|\alpha\psi\psi_t\right|^{\frac{q}{q-1}}\,dx dt\right]^{\frac{q-1}{q}}\right\}
\end{aligned}
$$
This proves the thesis.
\end{proof}

We now need to introduce the so called \textit{Whitney distance} $\delta : \Omega \to \R^+$, which is a function in $C^{\infty}(\Omega)$, regardless of the regularity of $\partial\Omega$, such that for all $x\in\Omega$
\begin{equation}\label{eq74b}
\begin{aligned}
&c_0^{-1}\,d(x)\le \delta(x) \le c_0\,\,d(x)\,,\\
&|\nabla\delta(x)|\le c_0\,,\\
&|\Delta\delta(x)|\le c_0\,\delta^{-1}(x)\,,
\end{aligned}
\end{equation}
where $d(x)$ has been defined in \eqref{eq20} and $c_0>0$ is a constant. These properties of the Whitney distance can be found, e.g., in \cite{Ba,Stein}.

\begin{lemma}\label{lemma73}
Let $V\in L^1_{loc}(\Omega\times[0,\infty))$, $V>0$ a.e., and $u_0\in L^1_{loc}(\Omega)$, $u_0\ge 0$ a.e. Assume that there exists a nonincreasing function $f:(0,\varepsilon_0)\rightarrow[1,\infty)$ such that $\lim_{\varepsilon\rightarrow0^+}f(\varepsilon)=+\infty$ and such that
for every $\varepsilon>0$ small enough conditions \eqref{eq218} hold. Let $u\ge 0$ be a weak solution of problem \eqref{problema2}, then
\begin{equation}\label{eq77}
\int_0^{+\infty}\int_{\Omega}u^qV\,dx dt<+\infty
\end{equation}
\end{lemma}

\begin{proof}
For every $\varepsilon>0$ small enough, we consider a smooth function $g_{\varepsilon}:[0,\infty)\to\R$ such that $0\le g_{\varepsilon}\le 1$, $g_{\varepsilon}\equiv 1$ in $[\varepsilon,+\infty)$, $\operatorname{supp}g_{\varepsilon}\subset [\frac{\varepsilon}{2},+\infty)$, $0\leq g_{\varepsilon}'\le \frac{C}{\varepsilon}$ and $|g_{\varepsilon}''|\le\frac{C}{\varepsilon^2}$ for some constant $C>0$. We also introduce $\eta$ a smooth function such that $0\leq\eta\leq1$, $\eta\equiv 1$ in $[0,\frac{1}{2}f(\varepsilon)]$, $\operatorname{supp}\eta\subset[0,f(\varepsilon))$ and $-\frac{C}{f(\varepsilon)}\leq\eta'\leq0$. Now let
\begin{equation}\label{eq78}
\psi_{\varepsilon}(x,t):=\phi_{\varepsilon}(x)\,\eta(t),
\end{equation}
where
\begin{equation}\label{eq79}
\phi_{\varepsilon}(x):=g_{\varepsilon}(\delta(x))
\end{equation}
where $\delta$ is the Whitney distance introduced in \eqref{eq74b}. Observe that, due to \eqref{eq78}, \eqref{eq79} and \eqref{eq74b} for every $x\in\Omega$, $t\in[0,T)$ we have
\begin{equation}\label{eq710b}
\begin{aligned}
&|\nabla\psi_{\varepsilon}(x,t)|\,=\,|g'_{\varepsilon}(\delta(x))\eta(t)\nabla\delta(x)|\,\le\,\frac{C}{\varepsilon}, \\
&|\Delta\psi_{\varepsilon}(x,t)|\,=\,|g''_{\varepsilon}(\delta(x))\eta(t)|\nabla\delta(x)|^2+g'_{\varepsilon}(\delta(x))\eta(t)\Delta\delta(x)|\,\le\,\frac{C}{\varepsilon^2},
\end{aligned}
\end{equation}
for some constant $C>0$. Hence for every $x\in\Omega$, $t\in[0,T)$ we have
\begin{align}
\left|(\psi_{\varepsilon})_t\right|\,\le\,\frac{C}{f(\varepsilon)},&\qquad\qquad\left|\alpha(\alpha-1)|\nabla\psi_{\varepsilon}|^2+\alpha\psi_{\varepsilon}\Delta\psi_{\varepsilon}\right|^{\frac{q}{q-1}}\,\le\,\frac{C}{\varepsilon^{\frac{2q}{q-1}}}.  \label{eq711}
\end{align}
Let $\tilde\Omega_\varepsilon=\{x\in\Omega\,|\,\delta(x)\ge\varepsilon\}$ and note that by \eqref{eq74b} for every $r>0$ we have $$\tilde{\Omega}_r\subset\Omega_\frac{r}{c_0},\qquad\qquad\Omega_r\subset\tilde{\Omega}_\frac{r}{c_0}.$$
We now observe, applying Lemma \ref{lemma71} with the test function $\psi_{\varepsilon}$ defined in \eqref{eq78}, that
\begin{equation}\label{eq713}
\begin{aligned}
\int_0^{\frac{1}{2}f(\varepsilon)}\int_{\tilde{\Omega}_{\varepsilon}}u^q\,V\,dx dt&\le\,\,\int_0^{+\infty}\int_\Omega u^q\,\psi_{\varepsilon}^{\alpha}V\,dx dt \\
&\le\, C\left\{\int_0^{+\infty}\int_\Omega V^{-\frac{1}{q-1}}\,\psi_{\varepsilon}^{\alpha-\frac{2q}{q-1}}\left|\alpha(\alpha-1)|\nabla\psi_{\varepsilon}|^2+\alpha\psi_{\varepsilon}\Delta\psi_{\varepsilon}\right|^{\frac{q}{q-1}}\,dx dt\right. \\
&\left. +\int_0^{+\infty}\int_\Omega V^{-\frac{1}{q-1}}\,\psi_{\varepsilon}^{\alpha-\frac{2q}{q-1}}\left|\alpha\psi_{\varepsilon}(\psi_{\varepsilon})_t\right|^{\frac{q}{q-1}}\,dx dt\right\}\\
&=:C(I_1+I_2).
\end{aligned}
\end{equation}
Now, due to the definition of $\psi_{\varepsilon}$ in \eqref{eq78} and by \eqref{eq218} and \eqref{eq711}, for every small enough $\varepsilon>0$ we have
\begin{equation}\label{eq714}
\begin{aligned}
I_1&\le \int_{0}^{f(\varepsilon)}\int_{\tilde{\Omega}_{\frac{\varepsilon}{2}}\setminus\tilde{\Omega}_{\varepsilon}}V^{-\frac{1}{q-1}}\left[\left|\alpha(\alpha-1)|\nabla\psi_{\varepsilon}|^2+\alpha\psi_{\varepsilon}\Delta\psi_{\varepsilon}\right|\right]^{\frac{q}{q-1}}\,dx dt\\
&\le\frac{C}{\varepsilon^{\frac{2q}{q-1}}}\, \int_{0}^{f(\varepsilon)}\int_{\Omega_{\frac{\varepsilon}{2c_0}}\setminus\Omega_{c_0\varepsilon}}V^{-\frac{1}{q-1}}\,dx dt\\
&\le\frac{C}{\varepsilon^{\frac{2q}{q-1}}}\,\sum_{k=0}^{N} \int_{0}^{f(\varepsilon)}\int_{\Omega_{\frac{2^{k-1}\varepsilon}{c_0}}\setminus\Omega_{\frac{2^{k}\varepsilon}{c_0}}}V^{-\frac{1}{q-1}}\,dx dt\\
&\le\frac{C}{\varepsilon^{\frac{2q}{q-1}}}\,\sum_{k=0}^{N} \left(\frac{2^{k}\varepsilon}{c_0}\right)^\frac{2q}{q-1}\,\,\leq\,\,C,
\end{aligned}
\end{equation}
where we set $N=[2\log_2c_0]+1$. Similarly, due to \eqref{eq78} and by \eqref{eq218} and \eqref{eq711}, we have
\begin{equation}\label{eq715}
\begin{aligned}
I_2&\le \frac{C}{(f(\varepsilon))^\frac{q}{q-1}}\int_{\frac{1}{2}f(\varepsilon)}^{f(\varepsilon)}\int_{\Omega_{\frac{\varepsilon}{2}}}V^{-\frac{1}{q-1}}\,dx dt\,\le\,C.
\end{aligned}
\end{equation}
By substituting \eqref{eq714} and \eqref{eq715} into \eqref{eq713} and letting $\varepsilon\to 0$ we obtain the thesis.
\end{proof}

We are now ready to prove Theorem \ref{teorema4}.
\begin{proof}[Proof of Theorem \ref{teorema4}]
For small enough $\varepsilon>0$ consider the test function $\psi_\varepsilon$ defined in \eqref{eq78}. Define
\begin{equation}\label{eq716}
K_{\varepsilon}:=\tilde{\Omega}_{\varepsilon}\times\left[0,\frac{1}{2}f(\varepsilon)\right]\,;
\end{equation}
and
\begin{equation}\label{eq717}
S_{K_{\varepsilon}}:=\left(\Omega\times\left[0,+\infty\right)\right)\setminus K_{\varepsilon}\,.
\end{equation}
Observe that $\psi_{\varepsilon}\equiv 1$ on $K_{\varepsilon}$, hence we can apply Lemma \ref{lemma72} with the test function $\psi_{\varepsilon}$ and we have
\begin{equation}\label{eq718}
\begin{aligned}
\int_0^{\frac{1}{2}f(\varepsilon)}\int_{\tilde{\Omega}_{\varepsilon}}u^q\,V\,dx dt\,&\le \int_0^{+\infty}\int_{\Omega}u^q\,\psi_{\varepsilon}^{\alpha}V\,dx dt \\
&\le C\left(\int\int_{S_{K_{\varepsilon}}}u^qV\psi^{\alpha}\,dxdt\right)^{\frac 1q}\\&\times
\left\{\left[\int\int_{S_{K_{\varepsilon}}}V^{-\frac{1}{q-1}}\psi_\varepsilon^{\alpha-\frac{2q}{q-1}}\left|\alpha(\alpha-1)|\nabla\psi_\varepsilon|^2+\alpha\psi_\varepsilon\Delta\psi_\varepsilon\right|^{\frac{q}{q-1}}\,dxdt\right]^{\frac{q-1}{q}}\right.\\
&\left.+\left[\int\int_{S_{K_{\varepsilon}}}V^{-\frac{1}{q-1}}\psi_\varepsilon^{\alpha-\frac{2q}{q-1}}\left|\alpha\psi_\varepsilon(\psi_\varepsilon)_t\right|^{\frac{q}{q-1}}\,dxdt\right]^{\frac{q-1}{q}}\right\} \\
&=:C(I_1+I_2)\left(\int\int_{S_{K_{\varepsilon}}}u^qV\psi_\varepsilon^{\alpha}\,dxdt\right)^{\frac 1q}.
\end{aligned}
\end{equation}
Now we can argue as in Lemma \ref{lemma73} and prove that there exists $C>0$ such that
$$
I_1\le C\,,\quad\quad I_2\le C\,.
$$
Thus we have
$$
\int_0^{\frac{1}{2}f(\varepsilon)}\int_{\tilde{\Omega}_{\varepsilon}}u^q\,V\,dx dt\,\le\, C\left(\int\int_{S_{K_{\varepsilon}}}u^qV\,dxdt\right)^{\frac 1q}\,.
$$
Letting $\varepsilon\to0$ by Lemma \ref{lemma73} we obtain
\begin{equation}\label{eq719}
 \int_0^{+\infty}\int_{\Omega}u^q\,V\,dx dt\,\,=0\,,
\end{equation}
which proves the thesis.
\end{proof}

\begin{proof}[Proof of Corollary \ref{corollario4}] By \eqref{eq37} and the assumptions on $f$, for $\varepsilon>0$ small enough we have
\begin{equation*}
\begin{aligned}
\int_0^{f(\varepsilon)}\int_{\Omega_\frac{\varepsilon}{2}\setminus\Omega_{\varepsilon}}V^{-\frac{1}{q-1}}\,dx dt&\le\, C f(\varepsilon)\int_{\Omega_\frac{\varepsilon}{2}\setminus\Omega_{\varepsilon}}d(x)^{\frac{q+1}{q-1}}f(d(x))^{-1}\,dx\\
&\le\, C \varepsilon^{\frac{q+1}{q-1}}\int_{\Omega_\frac{\varepsilon}{2}\setminus\Omega_{\varepsilon}}dx\,\,\leq\,C\varepsilon^{\frac{2q}{q-1}}
\end{aligned}
\end{equation*}
and
\begin{equation*}
\begin{aligned}
\int_{\frac{1}{2}f(\varepsilon)}^{f(\varepsilon)}\int_{\Omega_\frac{\varepsilon}{2}}V^{-\frac{1}{q-1}}\,dx dt&\le\, C f(\varepsilon)\int_{\Omega_\frac{\varepsilon}{2}}d(x)^{\frac{q+1}{q-1}}f(d(x))^{-1}\,dx\\
&\le\, C f(\varepsilon)\,\,\leq\,C f(\varepsilon)^{\frac{q}{q-1}}.
\end{aligned}
\end{equation*}
Thus conditions \eqref{eq218} are satisfied and by Theorem \ref{teorema4} $u\equiv0$ a.e. in $\Omega\times[0,\infty).$
\end{proof}

%
% ---- Bibliography ----
%

\bigskip
\bigskip
\bigskip

\textbf{Acknowledgements.} The authors are members of the Gruppo Nazionale per l'Analisi Mate\-ma\-ti\-ca, la Probabilit\`{a}
e le loro Applicazioni (GNAMPA) of the Istituto Nazionale di Alta Matematica (INdAM). D.D.M. and F.P. are partially supported by 2020 GNAMPA project ``Equazioni Ellittiche e Paraboliche ed Analisi Geometrica''. F.P. is supported by the PRIN-201758MTR2 project ``Direct and inverse problems for partial differential equations: theoretical aspects and applications.''

\end{document}